\documentclass{amsart}
\input{commands.tex}
\usepackage{hyperref}

\def\fr{\mathfrak f} 
\def\un{\mathfrak u} 
\def\frh{\fr^h} 
\def\unh{\un^h} 
\def\comm#1{}

\remfalse
\begin{document}

\title[Non-abelian zeta functions]{Moduli spaces of stable pairs and non-abelian zeta functions of curves via wall-crossing}

\author{Sergey Mozgovoy}%
\author{Markus Reineke}%
\email{mozgovoy@maths.tcd.ie}
\email{mreineke@uni-wuppertal.de}
\date{\today}

\begin{abstract}
In this paper we study and relate the non-abelian zeta functions introduced by Weng and invariants of the moduli spaces of arbitrary rank stable pairs over curves. We prove a wall-crossing formula for the latter invariants and obtain an explicit formula for these invariants in terms of the motive of a curve. Previously, formulas for these invariants were known only for rank 2 due to Thaddeus and for rank 3 due to Mu\~noz. Using these results we obtain an explicit formula for the non-abelian zeta functions, we check the uniformity conjecture by Weng for the ranks 2 and 3, and we prove the counting miracle conjecture. 
\end{abstract}

\maketitle

\section{Introduction}
This paper has two motivations.
The first one is the study of motivic invariants (like Poincar\'e polynomials, Hodge polynomials, or motives) of moduli spaces of pairs on a smooth projective curve. The moduli spaces of pairs were studied extensively in the last two decades \cite{bradlow_special,thaddeus_stable,huybrechts_stable,garcia-prada_invariant}. Their Poincar\'e resp.\ Hodge polynomials were computed by Thaddeus \cite{thaddeus_stable} in the rank two case and by Mu{\~n}oz \cite{munoz_hodgea} in the rank three case. For rank four it was proved \cite{munoz_motives}, and conjectured for general rank, that the motive of the moduli space can be expressed in terms of the motive of the curve. We will compute the motives of these moduli spaces for arbitrary rank in terms of an explicit Zagier-type formula, and in particular confirm the above conjecture.

Our second motivation is the work of Weng \cite{weng_zeta} on the (pure) non-abelian zeta functions of curves. Given a curve $X$ over a finite field $\bF_q$, let $\bfM(r,d)$ denote the set of isomorphism classes of semistable vector bundles on $X$ having rank $r$ and degree $d$. Define the rank $r$ pure non-abelian zeta function by
$$Z_{X,r}(t)=\sum_{k\ge0}\sum_{E\in\bfM(r,kr)}\frac{q^{h^0(X,E)}-1}{\n{\Aut E}}t^{k}.$$
\comm{Weng considers $t^{kr}$ instead of $t^k$.}
The special uniformity conjecture of Weng \cite[Conj.~9]{weng_zeta} suggests that the rank $r$ pure zeta function coincides with the group zeta function associated to the special linear group $\SL_r$ \cite[\S2]{weng_zeta}. This conjecture was announced to be a theorem in \cite[Theorem 5]{weng_special}. This result can be used to express the rank $r$ pure zeta functions in terms of the usual zeta function of a curve. We will use a different approach based on moduli spaces of pairs to compute rank $r$ zeta functions by an explicit Zagier-type formula. We will also check the uniformity conjecture for the rank $2$ and $3$ zeta functions.

Let us now describe our results in more detail. Let $X$ be a smooth projective complex curve. A pair $(E,s)$ on $X$ consists of a vector bundle $E$ on $X$ and a nonzero section $s\in H^0(X,E)$. There is a notion of stability of such pairs depending on a parameter $\ta\in\bR$. For any $(r,d)\in\bZ_{>0}\xx\bZ$, let $\cM_\ta(r,d)$ denote the moduli stack of \ta-semistable pairs with a vector bundle having rank $r$ and degree $d$ and let \[\fr_\ta(r,d)=(q-1)q^{(1-g)\binom r2}[\cM_\ta(r,d)]\]
be its motive up to some factor, where $q$ denotes the Lefschetz motive. Define the generating function
$$\fr_\ta=\sum_{r,d}\fr_\ta(r,d)x_1^rx_2^dx_3$$
in a certain completion of a skew polynomial ring. For example, for $\ta\gg0$ (we write $\ta=\infty$) we have
$$\fr_{\infty}=x_1x_3Z_X(x_2),$$
where $Z_X(t)$ is the motivic zeta function of $X$.

On the other hand, let $\un_{\ge\ta}(r,d)$ be the twisted motive of the moduli stack of vector bundles having rank $r$, degree $d$ and such that the quotients of their Harder-Narasimhan filtrations have slopes $\ge\ta$. These invariants can be computed by the formula of Zagier \cite{zagier_elementary}, based on the works of Harder and Narasimhan \cite{harder_cohomology}, Desale and Ramanan \cite{desale_poincare}, and Atiyah and Bott \cite{atiyah_yang}. Define the generating function
$$\un_{\ge\ta}=\sum_{d/r\ge\ta}\un_{\ge\ta}(r,d)x_1^rx_2^d.$$
Our first main result is the following wall-crossing formula (see Theorem \ref{th:main}):

\begin{theorem}
For any $\ta\in\bR$, we have
$$\fr_\ta=(\un_{>\ta}\inv\circ\fr_\infty\circ\un_{\ge\ta})|_{\mu\le\ta},$$
where the truncation $|_{\mu\le\ta}$ means that we keep only the coefficients $x_1^rx_2^dx_3$ with $\frac dr\le\ta$.
\end{theorem}

This result implies that the motive of $\cM_\ta(r,d)$ can be expressed in terms of the motive of $X$ and its symmetric products. This was conjectured in \cite{munoz_motives}.
Using generalizations of Zagier's formula for the motive of the moduli stack of semistable bundles (to be discussed in an appendix which also contains a new proof of Zagier's original formula), this yields the following explicit formula for the motive $[M_\tau(r,d)]$ of the moduli space of \ta-semistable pairs (see Theorem \ref{th:explicit}):

\begin{theorem} For $r\ge2$ and generic $\tau$, we have
\begin{multline*}
[M_\tau(r,d)]
=q^{(g-1)\binom r2}\sum_{r_1+\ldots+r_k=r-1}\frac{b_{r_1}\ldots b_{r_k}}{\prod_{i=1}^{k-1}{(1-q^{r_i+r_{i+1}})}}
\coeff_{t^{d-\lceil(r-1)\tau\rceil}}\\
\rbr{Z_X(t)\cdot
\rbr{\frac{q^{{F}_0}}{1-q^{r_1+1}t}-\sum_{p=1}^{k-1}\frac
{q^{{F}_p}(1-q^{r_p+r_{p+1}})t^{\delta_p}}
{(1-q^{r_{p+1}+1}t)(1-q^{-r_p}t)}-\frac{q^{{F}_k}}{1-q^{-r_k}t}}},
\end{multline*}
where $b_r$ equals (up to a twist) the motive of the moduli stack of rank $r$ bundles, and $F_p$ and $\delta_p$ are certain explicit exponents.
\end{theorem}

The wall-crossing formula can also be used to compute the higher zeta functions.
We can write the motivic version of the higher zeta-functions as follows
$$Z_{X,r}(t)=(q-1)\sum_{k\ge0}[\cM_k(r,kr)]t^k
=q^{(g-1)\binom r2}\sum_{k\ge0}\fr_k(r,kr)t^k
.$$
This means that in order to find $Z_{X,r}(t)$ we have to compute $\fr_\ta(r,d)$ for $\ta=\frac dr$. Applying the above theorem we obtain, for any $\ta\in\bR$:
$$\sum_{\frac dr=\ta}\fr_\ta(r,d)x_1^rx_2^dx_3=(\un_{>\ta}\inv\circ\fr_\infty\circ\un_{\ge\ta})|_{\mu=\ta}.$$
The following result describes the higher zeta functions explicitly.

\begin{theorem}
Let $\what Z_{X,r}(t)=t^{1-g}Z_{X,r}(t)$ and $\what Z_X(t)=t^{1-g}Z_X(t)$.
Then
\begin{multline*}
\what Z_{X,r}(t)=q^{(g-1)\binom r2}\sum_{r_1+\dots+r_k=r-1}\frac{b_{r_1}\dots b_{r_k}}{\prod_{i=1}^{k-1}(1-q^{r_i+r_{i+1}})}\\
\rbr{\frac{\what Z_X(t)}{1-q^{r_1+1}t}
-\sum_{i=1}^{k-1}\frac{(1-q^{r_i+r_{i+1}})q^{r_{<i}}t\what Z_X(q^{r_{\le i}}t)}{(1-q^{r_{<i}}t)(1-q^{r_{\le i+1}+1}t)}-\frac{q^{r_{<k}}t\what Z_X(q^{r-1}t)}{1-q^{r_{<k}}t}}.
\end{multline*}
\end{theorem}

This theorem implies a generalization of the counting miracle conjecture of Weng \cite[Conj.~15]{weng_zetaa}
\[q^{(1-g)r}Z_{X,r}(0)
=[\cM(r-1,0)],\]
where $\cM(r,0)$ denotes the moduli stack of semistable vector bundles having rank $r$ and degree zero.

Our approach can be also used in order to find the higher zeta functions of curves over finite fields. In this case motives should be substituted by the so-called c-sequences introduced in \cite{mozgovoy_poincare}. All of the above formulas remain the same.
\rem{Should we also formulate the Zagier-type formula for $Z_{X,r}(t)$ here? The special uniformity in low rank?}
\rem{In which form should we mention the section on DT invariants here?}

The reader should not be deceived by the apparent simplicity of our approach. A lot of obnoxious geometry happens behind the innocent algebraic scene. While for ranks $2$ and $3$ it is possible, with some effort, to control destabilizing loci when crossing the walls, the situation becomes much more complicated for higher ranks. Our basic idea goes back to the work of Thaddeus. In order to find the motivic invariant of the moduli stack $\cM_\ta(r,d)$ of \ta-semistable pairs, we first find this invariant for $\ta\gg0$ and then decrease \ta, thoroughly analyzing the behavior of our invariants when crossing the walls, \ie when \ta goes through the critical values, where some semistable pairs become non-semistable. In this way we can find $[\cM_\ta(r,d)]$ for any $\ta\ge\frac dr$. But in contrast to \cite{munoz_hodge,thaddeus_stable}, our approach does not use the geometry of the moduli spaces directly. Instead, we use ideas from motivic wall-crossing \cite{kontsevich_stability} and derive the behaviour of the motivic invariants from identities in a Hall algebra of a category of triples. For $\ta<\frac dr$ the moduli space is empty. One might ask why we do not cross just one wall at $\ta=\frac dr$ and find the invariant $[\cM_\ta(r,d)]$ for $\ta=\frac dr+\eps$ with $0<\eps\ll1$; the answer is that in order to prove the wall-crossing formula we need enough vanishing of second \Ext in the category of triples, which only holds for $\ta>\frac dr$.

{\bf Acknowledgments:} The authors would like to thank Tam\'as Hausel for helpful remarks about the results of the paper.

\section{Preliminaries}
All results of this section will be formulated for an algebraic curve $X$ over an algebraically closed field $\bk$ of characteristic zero, and for the motives of moduli stacks over it. Motives will be considered as elements in the Grothendieck ring $K_0({\rm St}_{\bk})$ of stacks (of finite type over $\bk$ and with affine stabilizers), which is related to the Grothendieck ring of $\bk$-varieties via localization or dimensional completion (see e.g. \cite{garcia-prada_motives}). We denote the Lefschetz motive by $q$ and always work in the coefficient ring $R=K_0({\rm St}_{\bk})[q^{\pm\frac{1}{2}}]$. We can also substitute motives by virtual Poincar\'e polynomials or E-polynomials. Also we can formulate all the results for a curve defined over a finite field, in which case we have to substitute motives by the so-called c-sequences introduced in \cite{mozgovoy_poincare}.

\subsection{Zeta function}
Given an algebraic variety $X$ we define its motivic zeta function
\eq{Z_X(t)=\sum_{n\ge0}[S^nX]t^n=\Exp(t[X]).}
If $X$ is a curve of genus $g$ then
$$Z_X(t)=\frac{P_X(t)}{(1-t)(1-qt)},$$
where $P_X(t)$ is a polynomial of degree $2g$. The value $P_X(1)$ equals the motive of the Jacobian $[\Jac X]$. The function $Z_X(t)$ satisfies the functional equation
$$Z_X(1/qt)=(qt^2)^{1-g}Z_X(t).$$
Therefore the function
\eq{\what Z_X(t)=t^{1-g}Z_X(t)}
satisfies $\what Z_X(1/qt)=\what Z_X(t)$.

\subsection{Stacks of bundles}
Let $X$ be a curve of genus $g$ and let $\al=(r,d)\in\bZ_{>0}\xx\bZ$.
Let $\Bun_{r,d}$ denote the stack of vector bundles over $X$ having rank $r$ and degree $d$. Its motive is independent of $d$ \cite[Section 6]{behrend_motivica}:
$$[\Bun_{r,d}]=\frac{[\Jac X]}{q-1}q^{(r^2-1)(g-1)}\prod_{i=2}^rZ_X(q^{-i})=\frac{P_X(1)}{q-1}\prod_{i=1}^{r-1}Z_X(q^i).$$
Define
\begin{equation}
b_r
=q^{(1-g)\binom r2}[\Bun_{r,d}]
=\frac{P_X(1)}{q-1}\prod_{i=1}^{r-1}\what Z_X(q^i)
.
\end{equation}
\comm{$[\Jac X]=P_X(1)$ (for $c$-sequences this means that $[\Jac X]_k=(\Jac X)(\bF_{q^k})=\prod_{i=1}^{2g}(1-\om_i^k)$).}
Let
\begin{enumerate}
	\item $\cM(\al)=\cM(r,d)$ be the moduli stack,
	\item $M(\al)=M(r,d)$ be the moduli space,
	\item $\bfM(\al)=\bfM(r,d)$ be the set of isomorphism classes
\end{enumerate}
of semistable vector bundles $E$ over $X$ with $\ch(E)=\al$.
Define
\begin{equation}
\be_{\al}
=q^{(1-g)\binom r2}[\cM(\al)].
\label{eq:}
\end{equation}

\subsection{Chern characters}
There is group homomorphism
\[\ch:K_0(\Coh X)\to\bZ^2\]
given by $\ch(E)=(\rk E,\deg E)$.
For any $E,F\in\Coh X$, define
\[\hi(E,F)=\dim\Hom_X(E,F)-\dim\Ext^1_X(E,F).\]
Let $\ch E=\al=(r,d)$ and $\ch F=\be=(r',d')$.
Then by the Riemann-Roch theorem
\[\hi(E,F)=rd'-r'd+(1-g)rr'.\]
Define
\eq{\hi(\al,\be)=rd'-r'd+(1-g)rr',}
\eq{\hi(\al)=\hi((1,0),\al)=d+(1-g)r,}
\eq{\ang{\al,\be}=\hi(\al,\be)-\hi(\be,\al)=2(rd'-r'd).}

\subsection{Integration map}
\label{prelim:int map}
Define the quantum affine plane $\bA_0$ to be the completion of the algebra $R[x_1,x_2^{\pm1}]$ with multiplication
$$x^{\al}\circ x^{\be}=(-q^\oh)^{\ang{\al,\be}}x^{\al+\be},$$
where we allow only elements $f=\sum_{\al\in\bN\xx\bZ} f_{\al}x^{\al}$ with
$$\inf\left\{\left.\frac d{r+1}\ \right|f_{r,d}\ne0\right\}>-\infty.$$
Let $H(\cA_0)$ be the Hall algebra of the category $\cA_0=\Coh X$ \cite{mozgovoy_poincare} (we use the opposite multiplication where the product $[E]\circ[F]$ counts extensions from $\Ext^1(F,E)$). There is an algebra homomorphism $I:H(\cA_0)\to\bA_0$ \cite{reineke_counting}, called an integration map, defined by
\[E\mto (-q^\oh)^{\hi(E,E)}\frac{x^{\ch E}}{[\Aut E]}.\]
For example, if $\one_\al\in H(\cA_0)$ (resp.\ $\one^\sst_\al\in H(\cA_0)$) is an element counting all (resp.\ all semistable) vector bundles having Chern character $\al=(r,d)$, then
\[I(\one_\al)=(-q^\oh)^{(1-g)r^2}[\Bun_{r,d}]x^\al=(-q^\oh)^{(1-g)r}b_r x^\al,\]
\[I(\one^\sst_\al)=(-q^\oh)^{(1-g)r^2}[\cM(\al)]x^\al=(-q^\oh)^{(1-g)r}\be_\al x^\al.\]
Using the Harder-Narasimhan filtrations and applying the integration map, we obtain
\eql{b_r
=\sum_{(\al_1,\dots,\al_k)\in\cP^d_\al}q^{\oh\sum_{i<j}\ang{\al_i,\al_j}}\be_{\al_1}\dots\be_{\al_k},}{eq:HN}
where $\cP^d_\al$ is the set of slope decreasing partitions of \al.

\subsection{Zagier formula}
\label{sec:zagier}
It was proved by Zagier \cite{zagier_elementary} that if there are families of elements $(b_r)_{r\ge1}$, $(\be_\al)_{\al\in\bZ_{>0}\xx\bZ}$ satisfying \eqref{eq:HN}, then
\eql{\be_\al
=\sum_{\over{r_1,\dots,r_k>0}{r_1+\dots+r_k=r}}
\left(\prod_{i=1}^{k-1}\frac{q^{(r_i+r_{i+1})\set{(r_1+\dots+r_i)d/r}}}{1-q^{r_i+r_{i+1}}}\right)
b_{r_1}\dots b_{r_k}.}{eq:zagier}
This gives an effective way to compute the motives of the moduli stacks $\cM(\al)$ of semistable vector bundles.

\section{Semistable pairs and triples}
\subsection{Semistable pairs}
Throughout the paper, let $X$ be a smooth projective curve over a field \bk, let $\ta\in\bR$ and $(r,d)\in\bZ_{>0}\xx\bZ$.

\begin{definition}
A pair $(E,s)$ over $X$ consists of a vector bundle $E$ over $X$ and a nonzero section $s\in H^0(X,E)$. Pairs over $X$ form a \bk-linear category: a morphism $f:(E,s)\to(E',s')$ between two pairs is an element $f=(f_0,f_1)\in\bk\xx\Hom_X(E,E')$ such that $f_1s=s'f_0$.
\end{definition}

\comm{
\begin{remark}[to move]
For the pair $(\cO_X,1)$ we want its algebra of endomorphisms to be $\bk$. Therefore, we define the group of morphisms from $(E_1,s)$ to $E_1',s')$ to consist of pairs $(f_0,f_1)\in\bk\xx\Hom(E_1,F_1)$ such that $f_1s=s f_0$. If $E=(E_1,\cO_X,s)$ is a a triple corresponding to $(E_1,s)$ then $\End E=\End(E_1,s)$.
\end{remark}
}

\begin{definition}
\label{df:ta stable pair}
A pair $(E,s)$ over $X$ is called \ta-semistable (resp.\ stable) if
\begin{enumerate}
	\item For any subbundle $F\sb E$ we have $\mu(F)\le\ta$ (resp.\ $\mu(F)<\ta$).
	\item For any subbundle $F\sb E$ with $s\in H^0(X,F)$ we have $\mu(E/F)\ge\ta$ (resp.\ $\mu(E/F)>\ta$).
\end{enumerate}
\end{definition}

\begin{definition}\label{df:generic}
Given $(r,d)\in\bZ_{>0}\xx\bZ$, we say that $\ta\in\bR$ is $(r,d)$-generic if $\ta\ne\frac dr$ and $\ta\notin\frac{1}{r'}\bZ$ for any $1\le r'< r$.
In this case any $\ta$-semistable pair $(E,s)$ with $\ch E=(r,d)$ is \ta-stable.
\end{definition}

We denote 
\begin{enumerate}
	\item by $\cM_\ta(r,d)$ the moduli stack (see \cite{garcia-prada_motives}),
	\item by $M_\ta(r,d)$ the moduli space (\cite{huybrechts_stable}),
	\item by $\bfM_\ta(r,d)$ the set of isomorphism classes
\end{enumerate}
of \ta-semistable pairs $(E,s)$ with $\ch(E)=(r,d)$, i.e. of rank $r$ and degree $d$.

\begin{remark}
If $\bk$ is algebraically closed of characteristic zero, we can define the motives $[\cM_\ta(r,d)]$, $[M_\ta(r,d)]$ as elements of $K_0({\rm St}_{\bk})$, the Grothendieck ring of stacks over $\bk$.
The analogue of $[\cM_\ta(r,d)]$ over a finite field $\bF_q$ is
$$\sum_{(E,s)\in\bfM_\ta(r,d)}\frac{1}{\n{\Aut(E,s)}}.$$
If $\ta$ is $(r,d)$-generic then
$$[\cM_\ta(r,d)]=\frac{[M_\ta(r,d)]}{q-1}.$$
\end{remark}

\begin{remark}
Let $(r,d)\in\bZ_{>0}\xx\bZ$ and let $\ta\in\bR$. 
If $M_\ta(r,d)\ne\es$ then $\ta\ge\frac dr$.
\end{remark}

\begin{lemma}
Let $(r,d)\in\bZ_{>0}\xx\bZ$ and let $\ta=\frac{d}{r}$. Then a pair $(E,s)$ with $\ch(E)=(r,d)$ is \ta-semistable if and only if $E$ is semistable.
\end{lemma}
\begin{proof}
Assume that $E$ is semistable. Then for any subbundle $F\sb E$ we have $\mu(F)\le\ta=\mu(E)$. Therefore $E$ is semistable.

Assume that $E$ is semistable. Then for any subbundle $F\sb E$ we have $\mu(F)\le\mu(E)=\ta$ and $\mu(E/F)\ge\mu(E)=\ta$. Therefore $(E,s)$ is \ta-semistable.
\end{proof}

\begin{corollary}
\label{cr:sec vs pairs}
Assume that $\bk=\bF_q$ and $\ta=\frac dr$. Then
$$\sum_{(E,s)\in\bfM_\ta(r,d)}\frac{1}{\n{\Aut(E,s)}}
=\frac{1}{q-1}\sum_{E\in\bfM(r,d)}\frac{q^{h^0(X,E)}-1}{\n{\Aut E}}.$$
\end{corollary}
\begin{proof}
Let $E\in\bfM(r,d)$. There is a natural action of the group $G_E=\bk^*\xx\Aut E$ on the set $M_E=H^0(X,E)\ms\set0$. The orbits of this action can be identified with the isomorphism classes of pairs $(E,s)$. The stabilizer of $s\in M_E$ can be identified with $\Aut(E,s)$. Therefore
$$\sum_{[s]\in M_E/G_E}\frac{1}{\n{\Aut(E,s)}}
=\sum_{s\in M_E}\frac1{\n{G_E}}
=\frac{\n{M_E}}{\n{G_E}}
=\frac{q^{h^0(X,E)}-1}{(q-1)\n{\Aut E}}.$$
\end{proof}


\subsection{The category of triples}

\begin{definition}
Let $Q$ be the quiver with two vertices $0,1$ and one arrow $s:0\to1$. We consider $Q$ as a category and define the category $\cT_X$ of triples on $X$ as the category of functors from $Q$ to $\Coh X$.
This is an abelian category.
An object $E\in\cT_X$ can be represented as a triple $(E_1,E_0,s_E)$ where $E_0,E_1\in\Coh X$ and $s_E\in\Hom_{\cO_X}(E_0,E_1)$.
\end{definition}

\comm{
Usually we consider only triples of the form $(E_1,V\ts\cO_X,s_E)$, where $E_1$ is a vector bundle and $V$ is a finite-dimensional vector space. We call them B-triples. We denote a triple of the form $(E_1,0,0)$ by $E_1$ and a triple of the form $(0,V\ts\cO_X,0)$ by $V$.
}

\begin{theorem}[{ \cite[Theorem 4.1]{gothen_homological}}]
Let $E,F\in\cT_X$ be two triples on the curve $X$.
Then there is a long exact sequence
\begin{multline*}
0\to\Hom(E,F)\to\bop_{i=0,1}\Hom_{\cO_X}(E_i,F_i)\to\Hom_{\cO_X}(E_0,F_1)\to\Ext^1(E,F)\to\\
\bop_{i=0,1}\Ext^1_{\cO_X}(E_i,F_i)\to\Ext^1_{\cO_X}(E_0,F_1)\to\Ext^2(E,F)\to0.
\end{multline*}
\end{theorem}

The following results about the vanishing of $\Ext^2$ in the category of triples are crucial for this paper. They will allow us to apply a Hall algebra formalism for the computation of motivic invariants.

\begin{proposition}
Let $E,F$ be two triples and assume that $s_E:E_0\to E_1$ is a monomorphism. Then $\Ext^2(E,F)=0$.
\end{proposition}
\begin{proof}
According to the previous theorem it is enough to show that
$$\Ext^1_{\cO_X}(E_1,F_1)\to\Ext^1_{\cO_X}(E_0,F_1)$$
is surjective. By Serre duality this is equivalent to the injectivity of
$$\Hom_{\cO_X}(F_1,E_0\ts\om_X)\to\Hom_{\cO_X}(F_1,E_1\ts\om_X)$$
which holds as $s_E:E_0\to E_1$ is a monomorphism.
\end{proof}

\begin{corollary}
\label{cr:vanishing}
Let $E,F$ be two triples and assume that one of the following conditions is satisfied
\begin{enumerate}
	\item $E_0=0$.
	\item $E_0=\cO_X$ and $s_E\ne0$.
\end{enumerate}
Then $\Ext^2(E,F)=0$.
\end{corollary}

\begin{definition}
For any $\si\in\bR$ and for any triple $E=(E_1,E_0,s_E)$ we define the \si-slope of $E$ by
$$\mu_\si(E)=\frac{\deg E_1+\deg E_0+\si\rk E_0}{\rk E_1+\rk E_0}\in\bR\cup\set\infty.$$
A triple $E$ is called semistable (resp.\ stable) with respect to $\mu_\si$ if for any proper nonzero subobject $F\sb E$ we have $\mu_\si(F)\le\mu_\si(E)$ (resp.\ $\mu_\si(F)<\mu_\si(E)$). 
\end{definition}

\comm{
\begin{definition}
Let $(r,d)\in\bZ_{>0}\xx\bZ$ and $\si\in\bR$. Denote by $\cN_\si(r,d)$ the moduli stack of $\si$-semistable $B$-triples $E=(E_1,\cO_X,s_E)$ with $s_E\ne0$ and $\ch(E_1)=(r,d)$. Similarly define the moduli space $N_\si(r,d)$ and the set of isomorphism classes $\bfN_\si(r,d)$.
\end{definition}
\begin{remark}
Note that we explicitly require that $s_E\ne0$ in the above definition. If $r,d,\si$ are as above and $E=(E_1,\cO_X,s_E)$ is \si-semistable with $E_1\ne0$ and $\ch(E_1)=(r,d)$ then $\si\ge \frac dr$. Moreover, if $\si>\frac dr$ or $E$ is \si-stable, then $s_E\ne0$.
\end{remark}
\begin{lemma}
Let $E_1$ be a vector bundle with $\ch E_1=(r,d)$. Let $\ta\in\bR$ and $\si=(r+1)\ta-d$. Then a pair $(E_1,s)$ is \ta-semistable if and only if the triple $(E_1,\cO_X,s)$ is semistable with respect to $\mu_\si$.
\end{lemma}
}

\section{Wall-crossing}

\subsection{Framed categories}

\begin{definition}
A framed category is a pair $(\cA,v)$, where $\cA$ is an abelian category and $v:K_0(\cA)\to\bZ$ is a group homomorphism such that $v(E)\ge0$ for any $E\in\cA$. For any $k\ge0$ we denote by $\cA_k$ the category of objects $E\in\cA$ with $v(E)=k$. The objects of the abelian category $\cA_0$ are called unframed objects. The objects of the category $\cA_1$ are called framed objects.
\end{definition} 

\begin{remark}
\label{rk:max unframed}
We assume that for any object $E\in\cA$ there exists a maximal unframed subobject $E_1\sb E$. Similarly, we assume that there exists a maximal unframed quotient $E\to E_2$.
\end{remark}

\begin{definition}
Let $(\cT,\cF)$ be a torsion pair on the category $\cA_0$ of unframed objects. A framed object $E\in\cA_1$ is called 
\begin{enumerate}
	\item $(\cT,\cF)$-stable if $E_1\in\cF$ and $E_2\in\cT$.
	\item $+\infty$-stable if it is $(0,\cA_0)$-stable, \ie if $E_2=0$.
	\item $-\infty$-stable if it is $(\cA_0,0)$-stable, \ie if $E_1=0$.
\end{enumerate}
\end{definition}

\begin{proposition}[Canonical filtration]
\label{pr:canonical filt}
Any framed object $E\in\cA_1$ has a unique filtration
$$E'\sb E''\sb E$$
such that $E'\sb\cT$, $E''/E'\in\cA_1$ is $(\cT,\cF)$-stable, and $E/E''\in\cF$.
\end{proposition}
\begin{proof}
Let $E$ be a framed object. We define $E'\sb E_1\sb E$ to be the torsion part of $E_1$ and we define $E''=\ker(E\to E_2\to E_2^f)$, where $E_2\to E_2^f$ is the free part of $E_2$. Uniqueness is left to the reader.
\end{proof}

\subsection{Framed category of triples}
Let $X$ be a curve. Let $\ang{\cO_X}$ be the subcategory of $\Coh X$ generated from $\cO_X$ by extensions. One can easily see that it is an abelian subcategory of $\Coh X$.
%
We define the category $\cA$ to be the category of triples $E=(E_1,E_0,s_E)$ such that $E_0\in\ang{\cO_X}$. We define the framing $v:K_0(\cA)\to\bZ$ by $v(E)=\rk E_0$. Then the category $\cA_0$ of unframed objects can be identified with the category $\Coh X$. Framed objects in $\cA$ have the form $(E_1,\cO_X,s_E)$, where $E_1\in\Coh X$ and $s_E\in\Hom_{\cO_X}(\cO_X,E_1)\iso H^0(X,E_1)$. 

\begin{remark}
If $E=(E_1,E_0,s_E)\in\cA$ then the maximal unframed subobject of $E$ is $(E_1,0,0)$ which we denote by $E_1$. The maximal unframed quotient of $E$ is $(\coker s_E,0,0)$ which we denote by $E_2$. This is in accordance with the conventions of Remark \ref{rk:max unframed}.
\end{remark}

\begin{definition}
\label{df:ta stable triple}
Let $\ta\in\bR$.
A framed object $E\in\cA_1$ is called \ta-semistable (resp.\ stable) if
\begin{enumerate}
	\item For any monomorphism $F\to E$ with unframed $F$ we have $\mu(F)\le \ta$ (resp.\ $<\ta$).
	\item For any epimorphism $E\to F$ with unframed $F$ we have $\mu(F)\ge \ta$ (resp.\ $>\ta$).
\end{enumerate}
A framed object $E\in\cA_1$ is called $\ta_+$-stable (resp.\ $\ta_-$-stable) if $E$ is $(\ta+\eps)$-semistable (resp.\ $(\ta-\eps)$-semistable) for $0<\eps\ll1$. This means that
\begin{enumerate}
	\item For any monomorphism $F\to E$ with unframed $F$ we have $\mu(F)\le \ta$ (resp.\ $<$).
	\item For any epimorphism $E\to F$ with unframed $F$ we have $\mu(F)>\ta$ (resp.\ $\ge$).
\end{enumerate}
\end{definition}

\begin{remark}
It is clear from Definitions \ref{df:ta stable pair},\ref{df:ta stable triple} that a pair $(E,s)$ is $\ta$-semistable if and only if the triple $(E,\cO_X,s)$ is \ta-semistable. Therefore the stack $\cM_\ta(r,d)$ can be identified with the moduli stack of framed \ta-semistable triples $E=(E_1,\cO_X,s_E)$ with $\ch E=(r,d)$ and $s_E\ne0$. In the next lemma we will see that the last condition $s_E\ne0$ is automatically satisfied for almost all \ta.
\end{remark}

\begin{lemma}
Let $E=(E_1,\cO_X,s_E)$ be a framed \ta-semistable object with $E_1\ne0$. Then $\mu(E_1)\le\ta$. If $\mu(E_1)<\ta$ then $s_E\ne0$.
\end{lemma}
\begin{proof}
Since $E_1$ is an unframed subobject of $E$, we have $\mu(E_1)\le\ta$. If $s_E=0$ then $E_1$ is an unframed quotient of $E$. Therefore $\mu(E_1)\ge\ta$, contradicting our assumption.
\end{proof}

\begin{lemma}
Let $E\in\cA_1$ be a framed object with $E_1\ne0$. Let $\ch E_1=(r,d)$ and $\ta\in\bR$. Then $E$ is \ta-semistable if and only if it is semistable with respect to $\mu_\si$, where $\si=(r+1)\ta-d$.
\end{lemma}
\begin{proof}
We note first that $\mu_\si(E)=\frac{d+\si}{r+1}=\ta$. Therefore $E$ is semistable with respect to $\mu_\si$ if and only if for any unframed $F\sb E$ we have $\mu(F)=\mu_\si(F)\le\mu_\si(E)=\ta$ and for any unframed quotient $E\to F$ we have $\mu(F)=\mu_\si(F)\ge\mu_\si(E)=\ta$. This is equivalent to \ta-stability of $E$.
\end{proof}

Define the category $\cA_{\ge\ta}$ to be the category of sheaves $E\in\cA_0=\Coh X$ such that the quotients of their Harder-Narasimhan filtration have slope $\ge\ta$. Similarly we define the categories $\cA_{\le\ta},\cA_{>\ta},\cA_{<\ta}$. The pairs of categories $(\cA_{>\ta},\cA_{\le\ta})$ and $(\cA_{\ge\ta},\cA_{<\ta})$ are torsion pairs in $\cA_0$.

\begin{lemma}
Let $E\in\cA_1$ be a framed object. Then
\begin{enumerate}
	\item $E$ is $\ta$-semistable \iff $E_1\in\cA_{\le\ta}$, $E_2\in\cA_{\ge\ta}$.
	\item $E$ is $\ta$-stable \iff $E_1\in\cA_{<\ta}$, $E_2\in\cA_{>\ta}$.
	\item $E$ is $\ta_+$-stable \iff 
	$E_1\in\cA_{\le\ta}$, $E_2\in\cA_{>\ta}$\iff
	$E$ is $(\cA_{>\ta},\cA_{\le \ta})$-stable.
	\item $E$ is $\ta_-$-stable \iff 
	$E_1\in\cA_{<\ta}$, $E_2\in\cA_{\ge\ta}$\iff
	$E$ is $(\cA_{\ge \ta},\cA_{< \ta})$-stable.
\end{enumerate}
\end{lemma}

The unique filtration of a framed object $E\in\cA_1$ with respect to the torsion pair $(\cA_{>\ta},\cA_{\le \ta})$ (see Prop.~\ref{pr:canonical filt}) will be called the canonical filtration with respect to $\ta_+$. 
The unique filtration of $E$ with respect to the torsion pair $(\cA_{\ge \ta},\cA_{< \ta})$ will be called the canonical filtration with respect to $\ta_-$.

\begin{lemma}
\label{lm:wc}
Let $E\in\cA$ be a framed $\ta$-semistable object with $s_E\ne0$. Let $\ch(E_1)=(r,d)$ and $\si=(r+1)\ta-d$ (\ie $\mu_\si(E)=\ta$). Then
\begin{enumerate}
	\item The canonical filtration of $E$ with respect to $\ta_+$ has the form $0=E'\sb E''\sb E$. If $E''\ne E$ then $E/E''$ is semistable and $\mu(E/E'')=\mu_\si(E'')=\ta$. If $s_E\ne0$ then $s_{E''}\ne0$.
	\item The canonical filtration of $E$ with respect to $\ta_-$ has the form $E'\sb E''=E$. If $E'\ne0$ then $E'$ is semistable and $\mu(E')=\mu_\si(E/E')=\ta$.
	If $\mu(E_1)<\ta$ then $s_{E/E'}\ne0$.
\end{enumerate}
\end{lemma}
\begin{proof}
1. Consider the canonical filtration $E'\sb E''\sb E$ with respect to $\ta_+$. Then $E'\sb\cA_{>\ta}$, while $\mu_\si(E)=\ta$. This implies $E'=0$. We have $E/E''\in\cA_{\le\ta}$, while if $E/E''\ne0$ then $\mu(E/E'')\ge\ta$. Therefore $\mu(E/E'')=\ta=\mu_\si(E'')$.
If $s_E\ne0$ then automatically $s_{E''}\ne0$.\\
2. Consider the canonical filtration $E'\sb E''\sb E$ with respect to $\ta_-$. Then $E/E''\in\cA_{<\ta}$, but $\mu_\si(E)=\ta$, and if $E/E''\ne0$ then $\mu(E/E'')\ge\ta$. This implies $E/E''=0$ and $E''=E$. We have $E'\in\cA_{\ge\ta}$, while if $E'\ne0$ then $\mu(E')\le\ta$. Therefore $\mu(E')=\ta=\mu_\si(E/E')$.
Assume that $\mu(E_1)<\ta$. 
The framed object $E/E'$ is $\si_-$-stable and therefore is indecomposable. To prove that $s_{E/E'}\ne0$ we have to show that $E'_1\ne E_1$. If $E'_1=0$ then we are done. If $E'_1\ne0$ then $\mu(E'_1)=\ta>\mu(E_1)$, so $E'_1\ne E_1$. 
\end{proof}

\begin{remark}
\label{rk:ext is zero}
In the first case of the previous lemma we have $\Ext^2(E/E'',E'')=0$ as $E/E''$ is unframed (see Corollary \ref{cr:vanishing}). In the second case of the previous lemma
we have $\Ext^2(E/E',E')=0$ if $\mu(E_1)<\ta$ (as $s_{E/E'}\ne0$ and we can apply Corollary \ref{cr:vanishing}).
\end{remark}

\begin{lemma}
\label{lm:wc inverse}
Let $E\in\cA_1$ be a framed object, $E'\sb E$ and $E''=E/E'$.
\begin{enumerate}
	\item If $E'\in\cA_1$ is $\ta_+$-stable, $E''$ is semistable and $\mu(E'')=\ta$, then $E$ is \ta-semistable.
	\item If $E''\in\cA_1$ is $\ta_-$-stable, $E'$ is semistable and $\mu(E')=\ta$, then $E$ is \ta-semistable.
\end{enumerate}
\end{lemma}
\begin{proof}
We prove just the first statment. Our assumption that $E'$ is
$\ta_+$-stable implies that $E'$ is \ta-semistable. Let $\si\in\bR$ be such that $\mu_\si(E'')=\ta$. Then both $E',E''$ are semistable with slope \ta with respect to $\mu_\si$. Therefore their extension $E$ is also semistable with slope \ta with respect to $\mu_\si$. This implies that $E$ is \ta-semistable.
\end{proof}

\section{Invariants}
\subsection{The class of a triple}
There is a group homomorphism $\cl:K_0(\cA)\to\bZ^3$ defined, for any $E=(E_1,E_0,s_E)\in\cA$, by
\[\cl(E)=(\rk E_1,\deg E_1,\rk E_0).\]
For any $E=(E_1,E_0,s_E)\in\cA$ and $F=(F_1,F_0,s_F)\in\cA$, define
\eq{\hi(E,F)=\sum_{k=0}^2(-1)^k\dim\Ext^k_\cA(E,F).}
Then
\eq{\hi(E,F)=\hi(E_0,F_0)+\hi(E_1,F_1)-\hi(E_0,F_1).}
Therefore, assuming $\cl E=\ub\al=(\al,v)$, $\cl F=\ub\be=(\be,w)$,
we obtain
\[\hi(E,F)
=(1-g)vw+\hi(\al,\be)-v\hi(\be)\]
Define
\eq{\hi(\ub\al,\ub\be)=(1-g)vw+\hi(\al,\be)-v\hi(\be),}
\eq{\ang{\ub\al,\ub\be}=
\hi(\ub\al,\ub\be)-\hi(\ub\be,\ub\al)
=\ang{\al,\be}-v\hi(\be)+w\hi(\al).}


\subsection{Integration map}
Define the quantum affine plane \bA to be the completion of the algebra $R[x_1,x_2^{\pm1},x_3]$ (as in Section \ref{prelim:int map}) with multiplication
$$x^{\al}\circ x^{\be}=(-q^\oh)^{\ang{\al,\be}}x^{\al+\be},$$
where we allow only elements $f=\sum_{\al\in\bN\xx\bZ\xx\bN} f_{\al}x^{\al}$ with
$$\inf\left\{\left.\frac d{r+1}\ \right|f_{r,d,v}\ne0\right\}>-\infty.$$
Let $H(\cA)$ be the Hall algebra of the category \cA. Define an integration map \mbox{$I:H(\cA)\to\bA$}
\[E=(E_1,E_0,s_E)\mto (-q^\oh)^{\hi(E,E)}\frac{x^{\cl E}}{[\Aut E]}.\]

\begin{remark}
\label{rk:I and Ext2}
This integration map restricts to an algebra homomorphism $I:H(\cA_0)\to\bA_0$ considered in Section \ref{prelim:int map}. Note, however, that $I:H(\cA)\to\bA$ is not 
an algebra homomorphism. But if $\Ext^2(F,E)=0$ then (see e.g.\ the proof of \cite[Lemma 3.3]{reineke_counting}) \[I([E]\circ[F])=I([E])\circ I([F]).\] 
\end{remark}

\subsection{The wall-crossing formula}
Let $\al=(r,d)\in\bZ_{>0}\xx\bZ$ and let $\ta\in\bR$. Recall that $\cM(\al)$ denotes the moduli stack of semistable vector bundles over $X$ having rank $r$ and degree $d$. Let
\eq{\un(\al)=(-q^\oh)^{\hi(\al,\al)+d}[\cM(\al)]=(-q^\oh)^{\hi(\al)}\be_\al}
be the motivic invariant \lqq counting\rqq (unframed) semistable bundles $E$ over $X$ with $\ch E=\al$.
Similarly, we define an invariant
\begin{equation}
\fr_\ta(\al)
=(q-1)(-q^\oh)^{\hi(\al,\al)-\hi(\al)+d}[\cM_\ta(\al)]
=(q-1)q^{(1-g)\binom r2}[\cM_\ta(\al)],
\end{equation}
\lqq counting\rqq framed $\ta$-semistable triples $E\in\cA$ with $s_E\ne0$ and $\ch E_1=\al$. Our main goal is to compute these invariants.

\begin{remark}
Let $\unh(\al)\in H(\cA_0)$ and $\frh_\ta(\al)\in H(\cA)$ be elements in the Hall algebras counting semistable vector bundles and framed \ta-semistable triples as above. Then
\[\un(\al)x^{\al}=(-q^\oh)^dI(\unh(\al)).\]
If $E$ is a triple with $\cl E=(\al,1)$ then $\hi(E,E)=(1-g)+\hi(\al,\al)-\hi(\al)$. This implies
\[\fr_\ta(\al)x^{(\al,1)}=(q-1)(-q^\oh)^{g-1+d}I(\frh_\ta(\al)).\]
\end{remark}

Define
\eq{\un_\ta=1+\sum_{\mu(\al)=\ta}\un(\al)x^\al\in\bA_0,\qquad \fr_{\ta}=\sum_{\al}\fr_\ta(\al)x^{(\al,1)}.}
We will see later that $\fr_\ta\in\bA$. Finally, define
\eq{\un_{\ge\ta}=\prod^{\curvearrowright}_{\ta'\ge\ta}\un_{\ta'},} where the product is taken in the decreasing slope order.

\comm{
\begin{remark}
Note that $\hi(\al,\al)-\hi(\al)=2(1-g)\binom r2-d$. Therefore
$$\fr_\ta(x_1,(-q^\oh)x_2,x_3)=\sum q^{(1-g)\binom r2}[\cM_\ta(\al)]x^{(\al,1)}.$$
\end{remark}
}

\begin{lemma}
If $\fr_\ta(r,d)\ne0$ then $0\le\frac dr\le\ta<\frac d{r-1}$.
\end{lemma}
\begin{proof}
We know from \cite[Theorem 6.1]{bradlow_stable} that if there exists a \ta-stable triple in $\cM_\ta(r,d)$ then $\frac dr<\ta<\frac d{r-1}$. Assume now that there exists a \ta-semistable triple $E\in \cM_\ta(r,d)$. Then, according to Lemma \ref{lm:wc}, there exists a $\ta_+$-stable framed object $E''\sb E$ with $\mu(E/E'')=\ta$. Let $\cl E''=(r'',d'',1)$ and $\ch(E/E'')=(r',d')$. Then
\[\frac{d''}{r''}\le\ta<\frac{d''}{r''-1},\qquad \frac{d'}{r'}=\ta.\]
This implies
\[\frac{d'+d''}{r'+r''}\le\ta<\frac{d'+d''}{r'+r''-1}.\]
Finally, the inequality $\frac dr\le\frac d{r-1}$ implies $d\ge0$.
\end{proof}

\begin{remark}
This lemma implies that $\fr_\ta$ is an element of \bA.
\end{remark}

\begin{remark}
Let $E$ be a framed $\infty$-semistable object with $\ch E_1=(r,d)$. This means that $\coker s_E$ is a finite sheaf. Therefore $r=1$, $d\ge0$ and $\coker s_E$ has length $d$. The endomorphism ring of $E$ equals $\bk$.
The moduli space $M_\infty(1,d)$ can be identified with a Hilbert scheme $\Hilb^dX\iso S^dX$.
Therefore
\begin{equation}\label{formulafinfty}
\fr_\infty
=x_1x_3\sum_{d\ge0}[S^dX]x_2^d
=x_1x_3Z_X(x_2).
\end{equation}
\end{remark}

For any series of the form $f=\sum_\al f_\al x^{(\al,1)}\in\bA$, define its truncation
$$f|_{\mu<\ta}=\sum_{\mu(\al)<\ta}f_\al x^{(\al,1)}.$$

\begin{theorem}
\label{th:main}
For any $\ta\in\bR$ we have
$$\fr_\ta=\left.\left(\un_{> \ta}\inv\circ\fr_{\infty}\circ \un_{\ge\ta}\right)\right|_{\mu\le \ta}.$$
In particular
$$\fr_{\ta_-}=\left.\left(\un_{\ge\ta}\inv\circ\fr_{\infty}\circ \un_{\ge\ta}\right)\right|_{\mu<\ta},\qquad
\fr_{\ta_+}=\left.\left(\un_{>\ta}\inv\circ\fr_{\infty}\circ \un_{>\ta}\right)\right|_{\mu\le\ta}.$$
\end{theorem}
\begin{proof}
Let
$$\unh_\ta=1+\sum_{\mu(\al)=\ta}\unh(\al),\qquad \frh_{\ta}=\sum_{\al}\frh_\ta(\al)$$
be the elements of the completed Hall algebras.
Note that $\frh_\ta(\al)=0$ if $\mu(\al)>\ta$. Therefore
$$\frh_{\ta}=\left.\frh_{\ta}\right|_{\mu\le\ta},\qquad
\frh_{\ta_+}=\left.\frh_{\ta_+}\right|_{\mu\le\ta},\qquad
\frh_{\ta_-}=\left.\frh_{\ta_-}\right|_{\mu<\ta}.$$
It follows from Lemma \ref{lm:wc} and Lemma \ref{lm:wc inverse} that
$$\frh_\ta=\frh_{\ta_+}\circ\unh_\ta,\qquad
\left.\frh_\ta\right|_{\mu<\ta}=\unh_\ta\circ \frh_{\ta_-}.$$
Applying the integration map and using Remarks \ref{rk:ext is zero} and \ref{rk:I and Ext2}, we obtain
$$\fr_\ta=\fr_{\ta_+}\circ\un_\ta,\qquad
\left.\fr_\ta\right|_{\mu<\ta}=\un_\ta\circ \fr_{\ta_-}.$$
This implies
$$\fr_{\ta_-}=\left.\left(\un_\ta\inv\circ\fr_{\ta_+}\circ \un_\ta\right)\right|_{\mu<\ta}.$$
Applying the same formula for all $\ta'\ge\ta$ we obtain
$$\fr_{\ta_-}=\left.\left(\un_{\ge\ta}\inv\circ\fr_{\infty}\circ \un_{\ge\ta}\right)\right|_{\mu<\ta}.$$
The other statements of the theorem are derived from this formula.
\end{proof}


\section{Zagier-type formula}

The wall-crossing formula Theorem \ref{th:main} for the motives of the moduli of stable pairs can now be made explicit, since explicit formulas for all three series are available. Namely, we have $\fr_\infty
=x_1x_3\sum_{d\ge0}[S^dX]x_2^d$ by formula (\ref{formulafinfty}); writing
$$\un_{\geq\tau}=1+\sum_{\mu(\alpha)\geq\tau}a_{\alpha}^{\geq\tau}(-q^\frac{1}{2})^{\chi(\alpha)}x^{(\alpha,0)},\;\;\;
\un_{>\tau}^{-1}=1+\sum_{\mu(\alpha)\geq\tau}c_{\alpha}^{>\tau}(-q^\frac{1}{2})^{\chi(\alpha)}x^{(\alpha,0)},$$
we have by Remark \ref{rk811} (replacing each sequence $(r_1,\ldots,r_k)$ by $(r_k,\ldots,r_1)$) and Remark \ref{rk814}:
$$a_{(r,d)}^{\geq\tau}=\sum_{r_1+\ldots+r_k=r}b_{r_1}\ldots b_{r_k}q^{-(r-r_1)d}\prod_{i=1}^{k-1}\frac{q^{(r_i+r_{i+1})\lceil r_{\geq i+1}\tau\rceil}}{1-q^{r_i+r_{i+1}}},$$
$$c_{(r,d)}^{>\tau}=-\sum_{r_1+\ldots+r_k=r}b_{r_1}\ldots b_{r_k}q^{(r-r_k)d}\prod_{i=1}^{k-1}\frac{q^{-(r_i+r_{i+1})\lfloor r_{\leq i}\tau\rfloor}}{1-q^{r_i+r_{i+1}}},$$
where $r_{\leq i}=r_1+\ldots+r_i$ and $r_{\geq i+1}=r_{i+1}+\ldots+r_k$.

For every $r\geq 2$, comparison of coefficients of $x^{(r,d,1)}$ in Theorem \ref{th:main} yields
\begin{multline*}
\fr_\tau(r,d)
=\sum_{e\geq 0}[S^eX]a_{(r-1,d-e)}^{\geq\tau} q^{d-re}
+\sum_{e\geq 0}[S^eX]c_{(r-1,d-e)}^{>\tau}q^{(1-g+e)(r-1)}\\
+\sum_{e\geq 0}\sum_{r-1=r'+r''}\sum_{d-e=d'+d''}[S^eX]c_{(r',d')}^{>\tau}b_{(r'',d'')}^{\geq\tau}q^{(1-g+e)r'-r''e+(r'+1)d''-r''d'}.
\end{multline*}
We insert the above formulas for $a_{\alpha}^{\geq\tau}$ and $c_{\alpha}^{>\tau}$ into this expression. First, this bounds the summation over $e$ to $e\leq d-(r-1)\tau$, resp.~to $e< d-(r-1)\tau$, in the first resp.~second sum. Second, the resulting summation over decompositions $r=r'+r''$, together with decompositions $r'=r_1'+\ldots+r_{k'}'$ and $r''=r_1''+\ldots+r_{k''}''$, can be replaced by the summation over decompositions $r-1=r_1+\ldots+r_k$, together with the choice of an index $p=1,\ldots,k-1$ which splits the latter decomposition into a part $(r_1,\ldots,r_p)=(r_1',\ldots,r_{k'}')$ and a part $(r_{p+1},\ldots,r_k)=(r_1'',\ldots,r_{k''}'')$. This gives the following formula for $\fr_\tau(r,d)$ (with $r_*$ indicating summation over decompositions of $r-1$ as before):

$$\sum_{e=0}^{\lfloor d-(r-1)\tau\rfloor}[S^eX]\sum_{r_*}b_{r_1}\ldots b_{r_k}q^{-(r-1-r_1)(d-e)+d-re}\prod_{i=1}^{k-1}\frac{q^{(r_i+r_{i+1})\lceil r_{\geq i+1}\tau\rceil}}{1-q^{r_i+r_{i+1}}}-$$

$$-\sum_{e=0}^{\lceil d-(r-1)\tau\rceil-1}[S^eX]\sum_{r_*}b_{r_1}\ldots b_{r_k}q^{(r-1-r_k)(d-e)+(1-g+e)(r-1)}\prod_{i=1}^{k-1}\frac{q^{-(r_i+r_{i+1})\lfloor r_{\leq i}\tau\rfloor}}{1-q^{r_i+r_{i+1}}}-$$

$$-\sum_{e\geq 0}\sum_{r_*}b_{r_1}\ldots b_{r_k}\sum_{d-e=d'+d''}\sum_{p=1}^{k-1}[S^eX]q^C
\prod_{i=1}^{p-1}\frac{q^{-(r_i+r_{i+1})\lfloor r_{\leq i}\tau\rfloor}}{1-q^{r_i+r_{i+1}}}
\prod_{i=p+1}^{k-1}\frac{q^{(r_i+r_{i+1})
\lceil r_{\geq i+1}\tau\rceil}}{1-q^{r_i+r_{i+1}}},$$

where
$$C=r_{\leq p-1}d'-r_{\geq p+2}d''+(1-g+e)r_{\leq p}-r_{\geq p+1}e+(r_{\leq  p}+1)d''-r_{\geq p+1}d'.$$

We consider the summation over $d'$ and $d''$. We can replace $d'=d-e-d''$; then $d''$ is bound by
$$r_{\geq p+1}\tau\leq d''<d-e-r_{\leq p}\tau,$$
and thus $e<d-(r-1)\tau$.
Analyzing the occurrences of $d'$ and $d''$ in the $q$-exponent $C$ above, this shows that the only part of the last sum depending on $d''$ is

$$\sum_{d''=\lceil r_{\geq p+1}\tau\rceil}^{\lceil d-e-r_{\leq p}\tau\rceil-1}q^{(r_p+r_{p+1}+1)d''}=\frac{q^{(r_p+r_{p+1}+1)\lceil r_{\geq p+1}\tau\rceil}-q^{(r_p+r_{p+1}+1)\lceil d-e-r_{\leq p}\tau\rceil}}{1-q^{r_p+r_{p+1}+1}}.$$

We can change the order of summation, summing over decompositions of $r-1$ and over $e$ first (adding one extra term for the case $d-(r-1)\tau\in\bN$), which gives:

\begin{lemma}\label{preliminaryexplicit} For $r\geq 2$, the motivic invariant $\fr_\tau(r,d)$ is given by
\begin{multline}\label{prelim}
\sum_{r_1+\dots+r_k=r-1}
\frac{b_{r_1}\ldots b_{r_k}}{\prod_{i=1}^{k-1}{(1-q^{r_i+r_{i+1}})}}
\Bigg([S^{d-(r-1)\tau}X]q^{A_0}\delta_{d-(r-1)\tau\in\bN}\\
+\sum_{e=0}^{\lceil d-(r-1)\tau\rceil -1}[S^eX]\cdot
\rbr{q^A-q^B-\sum_{p=1}^{k-1}q^{C_p}(q^{D_p}-q^{E_p})\frac{1-q^{r_p+r_{p+1}}}{1-q^{r_p+r_{p+1}+1}}}\Bigg),
\end{multline}
where\\ $A_0=(r-1)((r_1+1)\tau-d)+\sum_{i=1}^{k-1}(r_i+r_{i+1})\lceil r_{\geq i+1}\tau\rceil$,\\ $A=-(r-1-r_1)(d-e)+d-re+\sum_{i=1}^{k-1}(r_i+r_{i+1})\lceil r_{\geq i+1}\tau\rceil$,\\
$B=(r-1-r_k)(d-e)+(1-g+e)(r-1)-\sum_{i=1}^{k-1}(r_i+r_{i+1})\lfloor r_{\leq i}\tau\rfloor$,\\
$C_p=(1-g)r_{\leq p}+(r_{\leq p-1}-r_{\geq p+1})d+r_pe-\sum_{i=1}^{p-1}(r_i+r_{i+1})\lfloor r_{\leq i}\tau\rfloor+\sum_{i=p+1}^{k-1}(r_i+r_{i+1})\lceil r_{\geq i+1}\tau\rceil$,\\
$D_p=(r_p+r_{p+1}+1)\lceil r_{\geq p+1}\tau\rceil$, $E_p=(r_p+r_{p+1}+1)\lceil d-e-r_{\leq p}\tau\rceil$.
\end{lemma}

Consider the case of $(r,d)$-generic $\tau$, i.e. $\tau\ne d/r$ and $\tau\notin\frac1{r'}\bZ$ for all $1\le r'<r$.

\begin{theorem}\label{th:explicit}
For $r\geq 2$, $d\in\bZ$, and $(r,d)$-generic $\tau\in\bR$, we have
\begin{multline*}
[M_\tau(r,d)]=q^{(g-1)\binom r2}\sum_{r_1+\ldots+r_k=r-1}\frac{b_{r_1}\ldots b_{r_k}}{\prod_{i=1}^{k-1}{(1-q^{r_i+r_{i+1}})}}\coeff_{t^{d-\lceil(r-1)\tau\rceil}}\\
\rbr{Z_X(t)\cdot \rbr{\frac{q^{{F}_0}}{1-q^{r_1+1}t}-\sum_{p=1}^{k-1}\frac
{q^{{F}_p}(1-q^{r_p+r_{p+1}})t^{\delta_p}}
{(1-q^{r_{p+1}+1}t)(1-q^{-r_p}t)}-\frac{q^{{F}_k}}{1-q^{-r_k}t}}},
\end{multline*}
where
\begin{multline*}
F_p=(1-g)r_{\leq p}+(r_{\leq p}-r_{\geq p+1})d-r_p\lceil r_{\leq p}\tau\rceil+(r_{p+1}+1)\lceil r_{\geq p+1}\tau\rceil\\
-\sum_{i=1}^{p-1}(r_i+r_{i+1})\lfloor r_{\leq i}\tau\rfloor+\sum_{i=p+1}^{k-1}(r_i+r_{i+1})\lceil r_{\geq i+1}\tau\rceil
\end{multline*}
and $\delta_p$ equals $1$ if $\{r_{\leq p}\tau\}+\{r_{\geq p+1}\tau\}<1$ and zero otherwise.
\end{theorem}

\begin{proof} In the formula of the previous lemma, we perform summation over $e$ using the simple identity
$$\sum_{e=0}^{N}[S^eX]q^{ae}=\coeff_{t^{N}}\rbr{Z_X(t)\frac{q^{aN}}{1-q^{-a}t}}$$
for $N\geq 0$ and $a\in\bZ$. This simplifies the term $\sum_{e=0}^{d-\lceil(r-1)\tau\rceil}[S^eX]q^A$ in the above formula to $$\coeff_{t^{d-\lceil(r-1)\tau\rceil}}\rbr{Z_X(t)\frac{q^{-(r-1)d+(r_1+1)\lceil(r-1)\tau\rceil+\sum_{i=1}^{k-1}(r_i+r_{i+1})\lceil r_{\geq i+1}\tau\rceil}}{1-q^{r_1+1}t}};$$
note that the $q$-exponent equals $F_0$ if $r_0$ is interpreted as zero. Similarly we treat the term  $\sum_{e=0}^{d-\lceil(r-1)\tau\rceil}[S^eX]q^B$, interpreting $r_{k+1}$ as zero.

After some calculation, the term $\sum_{e=0}^{d-\lceil(r-1)\tau\rceil}[S^eX](q^{C_p+D_p}-q^{C_p+E_p})$ is rewritten as the $t^{d-\lceil(r-1)\tau\rceil}$-coefficient of
$$Z_X(t)q^{(1-g)r_{\leq p}+(r_{\leq p}-r_{\geq p+1})d-\sum_{i=1}^{p-1}(r_i+r_{i+1})\lfloor r_{\leq i}\tau\rfloor+\sum_{i=p+1}^{k-1}(r_i+r_{i+1})\lceil r_{\geq i+1}\tau\rceil}\cdot$$
$$\cdot\rbr{
\frac{q^{(r_p+r_{p+1}+1)\lceil r_{\geq p+1}\tau\rceil-r_p\lceil(r-1)\tau\rceil}}{1-q^{-r_p}t}-
\frac{q^{-(r_p+r_{p+1}+1)\lfloor r_{\leq p}\tau\rfloor+(r_{p+1}+1)\lceil(r-1)\tau\rceil}}{1-q^{r_{p+1}+1}t}}.$$
We use the following simple identity which holds for all $a,b,c,d\in\bZ$ and generic $\tau\in\bR$:
$$\frac{q^{(a+b)\lceil d\tau\rceil-a\lceil (c+d)\tau\rceil}}{1-q^{-a}t}-\frac{q^{-(a+b)\lfloor c\tau\rfloor+b\lceil(c+d)\tau\rceil}}{1-q^{b}t}=$$
$$=\frac{q^{b\lceil d\tau\rceil-a\lceil c\tau\rceil}(1-q^{a+b})t^{\delta(c,d,\tau)}}{(1-q^{-a}t)(1-q^bt)},$$
where $\delta(c,d,\tau)$ equals one if $\{c\tau\}+\{d\tau\}<1$ and zero otherwise. We apply this to $a=r_p$, $b=r_{p+1}+1$, $c=r_{\leq p}$ and $d=r_{\geq p+1}$ to simplify the previous expression to
$$\coeff_{t^{d-\lceil(r-1)\tau\rceil}}
\rbr{Z_X(t)
\frac{q^{F_p}t^{\delta(r_{\leq p},r_{\geq p+1},\tau)}}{(1-q^{r_{p+1}+1}t)(1-q^{-r_p}t)}}.$$

\end{proof}

\begin{remark} We can easily recover (slight variants of) the formulas of \cite{thaddeus_stable} and \cite{munoz_hodge} for stable pairs of rank two resp.~three. Namely, for generic $\tau$ we find
$$[M_\tau(2,d)]=\frac{P_X(1)}{q-1}\coeff_{t^{d-\lceil\tau\rceil}}
\rbr{Z_X(t)\cdot\rbr{\frac{q^{g-1-d+2\lceil\tau\rceil}}{1-q^2t}-\frac{q^{d-\lceil\tau\rceil}}{1-q^{-1}t}}}$$
and $[M_\tau(3,d)]=$
\begin{multline*}
=\frac{P_X(1)}{(1-q)^2(1-q^2)}\coeff_{t^{d-2\lceil\tau\rceil}}
\Bigg(Z_X(t)\cdot\Bigg(-P_X(q)
\rbr{\frac{q^{2g-2-2d+3\lceil 2\tau\rceil}}{1-q^3t}-\frac{q^{2d-2\lceil 2\tau\rceil}}{1-q^{-2}t}}\\
+P_X(1)\rbr{
\frac{q^{3g-3-2d+2\lceil 2\tau\rceil+2\lceil\tau\rceil}}{1-q^2t}
-\frac{q^{2g-2+\lceil\tau\rceil}(1-q^2)t^{2\lceil\tau\rceil-\lceil 2\tau\rceil}}{(1-q^2t)(1-q^{-1}t)}
-\frac{q^{g+1+2d-\lceil 2\tau\rceil-2\lceil\tau\rceil}}{1-q^{-1}t}
}\Bigg)\Bigg).
\end{multline*}
\end{remark}

\section{Non-abelian zeta functions}

\subsection{Non-abelian zeta functions}
Let us assume first that $X$ is a curve over a finite field $\bF_q$.
For any $r\ge1$ one defines the rank $r$ pure zeta function by the formula \cite[Def.~1]{weng_zeta} (we use $t$ instead of $t^r$ there)
\eq{Z_{X,r}(t)=\sum_{k\ge0}\sum_{E\in\bfM(r,kr)}\frac{q^{h^0(X,E)}-1}{\n{\Aut E}}t^{k}.}
By Corollary \ref{cr:sec vs pairs} we can write for $\ta=\frac dr$
$$\sum_{E\in\bfM(r,d)}\frac{q^{h^0(X,E)}-1}{\n{\Aut E}}
=(q-1)\sum_{(E,s)\in\bfM_\ta(r,d)}\frac{1}{\n{\Aut(E,s)}}
.$$
Therefore
$$Z_{X,r}(t)
=(q-1)\sum_{k\ge0}\sum_{(E,s)\in\bfM_k(r,kr)}\frac{1}{\n{\Aut(E,s)}}t^k.$$
If now $X$ is a curve over $\bk$, we can write the motivic version

\eq{Z_{X,r}(t)
=(q-1)\sum_{k\ge0}[\cM_k(r,kr)]t^k
=q^{(g-1)\binom r2}\sum_{k\ge0}\fr_k(r,kr)t^k
.}
We define also
\eq{\what Z_{X,r}(t)=t^{1-g}Z_{X,r}(t).}

The following properties of the non-abelian zeta functions were proved in \cite{weng_zeta}
\begin{enumerate}
	\item $Z_{X,1}(t)=Z_X(t)$.
	\item There exists a polynomial $P_{X,r}(t)$ of degree $2g$ such that
	\[Z_{X,r}(t)=\frac{P_{X,r}(t)}{(1-t)(1-q^rt)}.\]
	\item $\what Z_{X,r}(1/q^rt)=\what Z_{X,r}(t)$.
\end{enumerate}

\begin{remark}
Let us show that $Z_{X,1}(t)=Z_X(t)$. A triple $E=(E_1,\cO_X,s_E)$ with $s_E\ne0$ and $\ch E_1=(1,k)$ is $\ta$-semistable if and only if $k=\mu(E_1)\le\ta$. In particular, it is always $k$-semistable.
The moduli space $M_k(1,k)=M_\infty(1,k)$ can be identified with $\Hilb^kX=S^kX$. Therefore $[\cM_k(1,k)]=\frac{[S^kX]}{q-1}$ and
$$Z_{X,1}(t)=(q-1)\sum_{k\ge0}[\cM_k(1,k)]t^k=\sum_{k\ge0}[S^kX]t^k=Z_X(t).$$
\end{remark}

%

\subsection{Explicit formula}

\begin{theorem} For every $r\geq 2$, we have
\begin{multline*}
\what Z_{X,r}(t)=q^{(g-1)\binom r2}\sum_{r_1+\dots+r_k=r-1}\frac{b_{r_1}\dots b_{r_k}}{\prod_{i=1}^{k-1}(1-q^{r_i+r_{i+1}})}\\
\rbr{\frac{\what Z_X(t)}{1-q^{r_1+1}t}
-\sum_{i=1}^{k-1}\frac{(1-q^{r_i+r_{i+1}})q^{r_{<i}}t\what Z_X(q^{r_{\le i}}t)}{(1-q^{r_{<i}}t)(1-q^{r_{\le i+1}+1}t)}-\frac{q^{r_{<k}}t\what Z_X(q^{r-1}t)}{1-q^{r_{<k}}t}}.
\end{multline*}

\end{theorem}

\begin{proof} For $s\in\bN$, we apply Lemma \ref{preliminaryexplicit} to the case $\tau=s$ and $d=rs$; then all roundings are trivial, and using the simple identities
$$\sum_{i=1}^{p-1}(r_i+r_{i+1})r_{\leq i}=r_{\leq p-1}r_{\leq p},\;\;\; \sum_{i=p+1}^{k-1}(r_i+r_{i+1})r_{\geq i+1}=r_{\geq p+1}r_{\geq p+2}$$
we can simplify the exponents $A_0$ to $E_p$ to
$$A_0=0,\; A=(r_1+1)(s-e),\; B=(1-g+s)(r-1)-r_k(s-e),$$
$$C_p+D_p=(1-g+s)r_{\leq p}-r_p(s-e),$$
$$C_p+E_p=(1-g+s)r_{\leq p}+(r_{p+1}+1)(s-e).$$
To perform the summation $\sum_{s\geq 0}\mathcal{F}_s(r,rs)t^s$, we use the identity
$$\sum_{s\geq 0}\sum_{e=0}^{s-1}[S^eX]q^{as+b(s-e)}t^s=\frac{q^{a+b}tZ_X(q^at)}{1-q^{a+b}t}.$$
After some elementary calculations, we arrive at the claimed formula.
\end{proof}

In particular, we obtain

\eql{\what Z_{X,2}(t)=q^{g-1}b_1\rbr{\frac{\what Z_X(t)}{1-q^2t}-\frac{t\what Z_X(qt)}{1-t}},
}{eq:z2}

\begin{multline}
\what Z_{X,3}(t)=q^{3g-3}b_2
\rbr{\frac{\what Z_X(t)}{1-q^3t}-\frac{t\what Z_X(q^2t)}{1-t}}\\
+\frac{q^{3g-3}b_1^2}{1-q^2}
\rbr{\frac{\what Z_X(t)}{1-q^2t}
-\frac{(1-q^2)t\what Z_X(qt)}{(1-t)(1-q^3t)}
-\frac{qt\what Z_X(q^2t)}{1-qt}}.
\label{eq:z3}
\end{multline}

The following result generalizes the counting miracle conjecture \cite[Conj.~15]{weng_zetaa} for arbitrary genus curves.

\begin{corollary}
The element $\fr_0(r,0)=q^{(1-g)\binom r2}Z_{X,r}(0)$ is equal to $\be_{r-1,0}$.
\end{corollary}
\begin{proof}
We have
\[q^{(1-g)\binom r2}Z_{X,r}(0)=\sum_{r_1+\dots+r_k=r-1}
\frac{b_{r_1}\dots b_{r_k}}{\prod_{i=1}^{k-1}
(1-q^{r_i+r_{i+1}})}\]
and this is the formula for $\be_{r-1,0}$, see equation \eqref{eq:zagier}.
\end{proof}

\subsection{Special uniformity}
In this section we will verify the special uniformity conjecture of Weng \cite[Conj.~9]{weng_zeta} for rank $2$ and $3$ bundles.
Given a curve $X$ over $\bF_q$, one can define a zeta function
$\what\zeta^{\SL_r}_X(s)$
associated to the group $\SL_r$ (see \cite[\S3.1]{weng_special} and \cite[\S2.2]{weng_zeta}). For $r=2$, we have \cite[\S2.2]{weng_zetaa}
\begin{equation}
\what\zeta_X^{\SL_2}(-2s)
=\frac{\what\zeta_X(2s)}{1-q^{-2s+2}}+\frac{\what\zeta_X(2s-1)}{1-q^{2s}},
\label{eq:rk2 special}
\end{equation}
where $\what\zeta_X(s)=\what Z_X(q^{-s})$.
One can actually define the zeta function $\what Z^{\SL_r}_X(t)$ such that (note that we use $q^s$ and not $q^{-s}$ here)
\[\what\zeta^{\SL_r}_X(s)=\what Z^{\SL_r}_X(q^s).\]
Taking $t=q^{-2s}$ in \eqref{eq:rk2 special}, we obtain
\eql{\what Z^{\SL_2}_X(t)
=\frac{\what Z_X(t)}{1-q^2t}-\frac{t\what Z_X(qt)}{1-t}.}{eq:sl2}
Similarly, for $r=3$ the formula for $\what Z^{\SL_3}_X(t)$ is (see \cite[\S2.3]{weng_zetaa}, where we use $t=q^{-3s}$ and $\what\zeta_X(1)=\frac{P(1)}{q-1}$)
\begin{multline}
\what Z^{\SL_3}_X(t)
=\what Z_X(q^{-2})\rbr{\frac{\what Z_X(t)}{1-q^3t}-\frac{t\what Z_X(q^2t)}{1-t}}\\
+\frac{P_X(1)}{(q-1)(1-q^2)}
\rbr{\frac{\what Z_X(t)}{1-q^2t}-\frac{qt\what Z_X(q^2t)}{1-qt}} -\frac{P_X(1)}{q-1}{\frac{t\what Z_X(qt)}{(1-t)(1-q^3t)}}
\label{eq:sl3}
\end{multline}

On the other hand, let $\what\zeta_{X,r}(s)=\what Z_{X,r}(q^{-rs})$ and let $\al_{r,0}=Z_{X,r}(0)$.
The following conjecture was formulated by Weng

\begin{conjecture}[see {\cite[Conj.~9]{weng_zeta}} and {\cite[Theorem~5]{weng_special}}]
We have
\[\what\zeta_{X,r}(s)=c_{X,r}\what\zeta^{\SL_r}_X(-rs)\]
for some constant $c_{X,r}$.
\end{conjecture}

Using $t=q^{-rs}$, we can write it also in the form
\[\what Z_{X,r}(t)=c_{X,r}\what Z^{\SL_r}_X(t).\]
The proof of this conjecture was announced in \cite[Theorem 5]{weng_special}. At the moment there exists only a proof for an elliptic curve and small ranks \cite{weng_zetaa}. Let us check it for $r=2$ and $r=3$.

The formulas \eqref{eq:z2} and \eqref{eq:sl2} imply
\eq{\what Z_{X,2}(t)=q^{g-1}b_1 Z^{\SL_2}_X(t).}
The formulas \eqref{eq:z3}, \eqref{eq:sl3} and $b_2=b_1\what Z_X(q)=b_1\what Z_X(q^{-2})$ imply
\eq{\what Z_{X,3}(t)=q^{3g-3}b_1 Z^{\SL_3}_X(t).}
These equations suggest that the precise form of the uniformity conjecture should be
\eq{\what Z_{X,r}(t)=q^{(g-1)\binom r2}b_1\what Z^{\SL_r}_X(t).}

\section{Appendix. Slice and inversion formulas}

\subsection{Slice formula}
Let $\Ga$ be a lattice with a stability function $\mu$ and let $\Ga_+\subset\Ga$ be a semigroup such that any element in $\Ga_+$ has a finite number of partitions, that is, the set
\[\cP_\al=\sets{(\al_1,\dots,\al_k)}{\al_i\in\Ga_+,\,\sum\al_i=\al}\]
is finite for any $\al\in\Ga_+$. For example we could take $\Ga=\bZ^n$ and $\Ga_+=\bN^n\ms\set0$.
Let $R$ be a (non-commutative) ring and let
\[(a_\al)_{\al\in\Ga_+},\qquad (b_\al)_{\al\in\Ga_+}\]
be two families of elements in $R$ satisfying
\begin{equation}
b_\al=\sum_{\over{(\al_1,\dots,\al_k)\in\cP_\al}{\mu(\al_1)>\dots>\mu(\al_k)}}a_{\al_1}\dots a_{\al_k}\qquad \forall\al\in\Ga_+.
\label{eq:b as a}
\end{equation}

\begin{remark}
Usually we interpret the elements $a_\al$ as counting semistable objects having class \al and the elements $b_\al$ as counting arbitrary objects having class \al. Then the above formula corresponds to the Harder-Narasimhan filtration.
\end{remark}

It was proved in \cite{reineke_harder-narasimhan}, \cite[Theorem 3.2]{mozgovoy_poincare} that \eqref{eq:b as a} is equivalent to
\begin{equation}
a_\al=\sum_{\over{(\al_1,\dots,\al_k)\in\cP_\al}{\mu(\al'_i)>\mu(\al)\ \forall 1\le i<k}}(-1)^{k-1}b_{\al_1}\dots b_{\al_k}\qquad \forall \al\in\Ga_+,
\label{eq:a as b}
\end{equation}
where $\al'_i=\al_1+\dots+\al_i$. In this section we will prove similar formulas for the following \lqq slice invariants\rqq
\begin{equation}
a^{\le\ta}_\al=\sum_{\over{(\al_1,\dots,\al_k)\in\cP_\al}{\ta\ge\mu(\al_1)>\dots>\mu(\al_k)}}a_{\al_1}\dots a_{\al_k},\qquad
a^{\ge\ta}_\al=\sum_{\over{(\al_1,\dots,\al_k)\in\cP_\al}{\mu(\al_1)>\dots>\mu(\al_k)\ge\ta}}a_{\al_1}\dots a_{\al_k}
\label{eq:a ta}
\end{equation}
as well as similarly defined $a^{<\ta}_\al$, $a^{>\ta}_\al$ and
\begin{equation}
a^{[\ta',\ta]}_\al=\sum_{\over{(\al_1,\dots,\al_k)\in\cP_\al}{\ta\ge\mu(\al_1)>\dots>\mu(\al_k)\ge\ta'}}a_{\al_1}\dots a_{\al_k}
\label{eq:ta-ta'}
\end{equation}
defined for any $\ta,\,\ta'\in\bR$.

\begin{theorem}
\label{th:slices}
We have
\begin{enumerate}
	\item For $\mu(\al)\le\ta$
\[a^{\le\ta}_\al
=\sum_{\over{(\al_1,\dots,\al_k)\in\cP_\al}{\mu(\al'_i)>\ta\ \forall 1\le i<k}}(-1)^{k-1}b_{\al_1}\dots b_{\al_k}.\]
	\item For $\mu(\al)\ge\ta$
\[a_\al^{\ge\ta}
=\sum_{\over{(\al_1,\dots,\al_k)\in\cP_\al}{\mu(\al-\al'_i)<\ta\ \forall 1\le i<k}}(-1)^{k-1}b_{\al_1}\dots b_{\al_k}.\]
	\item For $\mu(\al)<\ta$
\[a^{<\ta}_\al
=\sum_{\over{(\al_1,\dots,\al_k)\in\cP_\al}{\mu(\al'_i)\ge\ta\ \forall 1\le i<k}}(-1)^{k-1}b_{\al_1}\dots b_{\al_k}.\]
	\item For $\mu(\al)>\ta$
\[a_\al^{>\ta}
=\sum_{\over{(\al_1,\dots,\al_k)\in\cP_\al}{\mu(\al-\al'_i)\le\ta\ \forall 1\le i<k}}(-1)^{k-1}b_{\al_1}\dots b_{\al_k}.\]
	\item For $\ta'\le\mu(\al)\le\ta$
\[a^{[\ta',\ta]}_\al
=\sum_{\over{(\al_1,\dots,\al_k)\in\cP_\al}{
\mu(\al'_i)>\ta\text{ or }\mu(\al-\al'_i)<\ta'\ \forall 1\le i<k
}}(-1)^{k-1}b_{\al_1}\dots b_{\al_k}.\]
\end{enumerate}
\end{theorem}
\begin{proof}
It is enough to prove the last formula.
We will denote the set of sequences appearing in the last sum by $\cP_\al(\ta,\ta')$. The sequences in
\[\cP_\al^\st=\sets{(\al_1,\dots,\al_k)\in\cP_\al}{\mu(\al'_i)>\mu(\al)\ \forall 1\le i<k}\]
will be called stable. Note that if $\ub\al\in\cP_\al^\st$ and $\ub\be\in\cP_\be^\st$ are stable sequences and $\mu(\al)>\mu(\be)$ then the concatenation $(\ub\al,\ub\be)$ is again stable.
Given a sequence $\ub\al=(\al_1,\dots\al_k)\in\cP_\al$, let $\cA_{\ub\al}(\ta,\ta')$ consist of all sets
\[S=\set{s_1<\dots<s_{r-1}}\subset\set{1,\dots,k-1}\]
such that the sequences (we let $s_0=0$, $s_r=k$)
\[(\al_{s_{i-1}+1},\dots,\al_{s_i})\]
are stable and the coarsening sequence $\ub\al_S=(\be_1,\dots,\be_r)$ of elements $\be_i=\al_{s_{i-1}+1}+\dots+\al_{s_i}$ satisfies
\[\ta\ge\mu(\be_1)>\dots>\mu(\be_r)\ge\ta'.\]
After the substitution of \eqref{eq:a as b} into \eqref{eq:ta-ta'} we can see that the formula we want to prove follows from
\begin{equation}
\sum_{S\in\cA_{\ub\al}(\ta,\ta')}(-1)^{\n S}=\case{0&\ub\al\notin\cP_\al(\ta,\ta')\\1&\ub\al\in\cP_\al(\ta,\ta').}
\label{eq:sum}
\end{equation}
If $\ub\al\in\cP_\al(\ta,\ta')$ then for any $i<k$ we have either $\mu(\al'_i)>\ta\ge\mu(\al)$ or $\mu(\al-\al'_i)<\ta'\le\mu(\al)$ which also implies $\mu(\al'_i)>\mu(\al)$. Therefore $\ub\al$ is stable. This implies that if \ub\al is non-stable then the right hand side of \eqref{eq:sum} is zero. On the other hand the left hand side of \eqref{eq:sum} is also zero as $\cA_{\ub\al}(\ta,\ta')$ is empty for non-stable $\ub\al$.

From now on we will assume that $\ub\al$ is stable.
Given an element $1\le i<k$, there exists $S\in\cA_{\ub\al}(\ta,\ta')$ with $i\in S$ if and only if $\set i\in\cA_{\ub\al}(\ta,\ta')$. Moreover, $\set i\in\cA_{\ub\al}(\ta,\ta')$ if and only if the sequences $(\al_1,\dots,\al_i)$ and $(\al_{i+1},\dots,\al_k)$ are stable and
\[\ta\ge\mu(\al'_i)>\mu(\al-\al'_i)\ge\ta'.\]

Assume that $\set i\notin\cA_{\ub\al}(\ta,\ta')$ for some $1\le i<k$. Consider the sequence
\[\ub\al'=(\al_1,\dots,\al_{i-1},\al_i+\al_{i+1},\dots,\al_k)\]
and note that $\cA_{\ub\al}(\ta,\ta')=\cA_{\ub\al'}(\ta,\ta')$.
We claim that $\ub\al\in\cP_\al(\ta,\ta')$ if and only if $\ub\al'\in\cP_\al(\ta,\ta')$. If this is false, then $\ub\al'\in\cP_\al(\ta,\ta')$, $\mu(\al'_i)\le\ta$ and $\mu(\al-\al'_i)\ge\ta'$. Therefore
\[\ta\ge\mu(\al'_i)>\mu(\al-\al'_i)\ge\ta'.\]
For any $1\le j<i$, we have either $\mu(\al'_j)>\ta\ge\mu(\al'_i)$ or $\mu(\al-\al'_j)<\ta'\le\mu(\al-\al'_i)$. The last inequality
implies
\[\mu(\al'_i-\al'_j)<\mu(\al-\al'_i)<\mu(\al'_i)\]
and therefore $\mu(\al'_j)>\mu(\al'_i)$. This implies that the sequence $(\al_1,\dots,\al_i)$ is stable. Similarly we can show that $(\al_{i+1},\dots,\al_k)$ is stable. Therefore $\set i\in\cA_{\ub\al}(\ta,\ta')$ contradicting to our assumption. This proves that $\ub\al\in\cP_\al(\ta,\ta')$ if and only if $\ub\al'\in\cP_\al(\ta,\ta')$ and we can apply induction to $\ub\al'$.

Assume that $\set i\in\cA_{\ub\al}(\ta,\ta')$ for any $1\le i<k$. 
Then the inequalities $\ta\ge\mu(\be_1)$, $\mu(\be_r)\ge\ta'$ in the definition of $\cA_{\ub\al}(\ta,\ta')$ are automatically satisfied and we can write $\cA_{\ub\al}(\ta,\ta')=\cA_{\ub\al}(\infty,-\infty)$.
Moreover, $\ub\al\in\cP_\al(\ta,\ta')$ if and only if $k=1$.
Now we apply \cite[Lemma 5.4]{reineke_harder-narasimhan}
\[\sum_{S\in\cA_{\ub\al}(\infty,-\infty)}(-1)^{\n S}=\case{0,&k>1,\\1,&k=1.}\]
\end{proof}

\subsection{Inversion formula}
\label{subs:inversion}
Consider the ring of power series $R\pser{\Ga_+}$ and its elements
\[b=1+\sum b_\al y^\al,\qquad a^{\le\ta}=1+\sum a^{\le\ta}_\al y^\al,\qquad a^{\ge\ta}=1+\sum a^{\ge\ta}_\al y^\al,\]
and similarly defined elements $a^{<\ta}$, $a^{>\ta}$. The elation \eqref{eq:b as a} implies $b=a^{\ge\ta}a^{<\ta}$.

\begin{theorem}
We have 
\begin{enumerate}
	\item $(a^{\le\ta})\inv=1+\sum_{\mu(\al)\le\ta} c^{\le\ta}_\al y^\al$, where
\[c^{\le\ta}_\al
=\sum_{\over{(\al_1,\dots,\al_k)\in\cP_\al}{\mu(\al-\al'_i)\le\ta\ \forall 1\le i<k}}(-1)^k b_{\al_1}\dots b_{\al_k}.\]
	\item $(a^{\ge\ta})\inv=1+\sum_{\mu(\al)\ge\ta} c^{\ge\ta}_\al y^\al$, where
	\[c^{\ge\ta}_\al
=\sum_{\over{(\al_1,\dots,\al_k)\in\cP_\al}{\mu(\al'_i)\ge\ta\ \forall 1\le i<k}}(-1)^k b_{\al_1}\dots b_{\al_k}.\]
	\item $(a^{<\ta})\inv=1+\sum_{\mu(\al)<\ta} c^{<\ta}_\al y^\al$, where
\[c^{<\ta}_\al
=\sum_{\over{(\al_1,\dots,\al_k)\in\cP_\al}{\mu(\al-\al'_i)<\ta\ \forall 1\le i<k}}(-1)^k b_{\al_1}\dots b_{\al_k}.\]
	\item $(a^{>\ta})\inv=1+\sum_{\mu(\al)>\ta} c^{>\ta}_\al y^\al$, where
	\[c^{>\ta}_\al
=\sum_{\over{(\al_1,\dots,\al_k)\in\cP_\al}{\mu(\al'_i)>\ta\ \forall 1\le i<k}}(-1)^k b_{\al_1}\dots b_{\al_k}.\]
\end{enumerate}
\end{theorem}
\begin{proof}
It is enough to prove the second formula.
Let $c^{\ge\ta}=1+\sum c^{\ge\ta}_\al y^\al$. To prove that $(a^{\ge\ta})\inv=c^{\ge\ta}$ we have to show that
\[(c^{\ge\ta})\inv a^{<\ta}=b\]
or equivalently $a^{<\ta}=c^{\ge\ta}b$.
We proved in Theorem \ref{th:slices} that, for $\mu(\al)<\ta$, we have
\[a^{<\ta}_\al=
\sum_{\over{(\al_1,\dots,\al_k)\in\cP_\al}{\mu(\al'_i)\ge\ta\ \forall 1\le i<k}}(-1)^{k-1} b_{\al_1}\dots b_{\al_k}.\]
Let us consider the coefficient of $y^\al$ in $c^{\ge\ta}b$. For $(\al_1,\dots,\al_k)\in\cP_\al$ the product $b_{\al_1}\dots b_{\al_k}$ occurs in this coefficient if and only only if $\mu(\al'_i)\ge\ta$ for $1\le i<k$. If $\mu(\al)<\ta$ then it occurs once as a summand of $c^{\ge\ta}_{\al-\al_k}\cdot b_{\al_k}$. Therefore the coefficient of $y^\al$ in $c^{\ge\ta}b$ coincides with $a^{<\ta}_\al$. If $\mu(\al)\ge\ta$ then the product $b_{\al_1}\dots b_{\al_k}$ occurs twice, with different signs: as a summand of $c^{\ge\ta}_\al\cdot 1$ and as a summand of $c^{\ge\ta}_{\al-\al_k}\cdot b_{\al_k}$. Therefore the coefficient of $y^\al$ in $c^{\ge\ta}b$ is zero.
\end{proof}

\begin{theorem}
Let $a^{[\ta'\ta]}=1+\sum_{\ta'\le\mu(\al)\le\ta}a_\al^{[\ta',\ta]}y^\al$ and \[(a^{[\ta',\ta]})\inv=c^{[\ta',\ta]}=1+\sum_{\ta'\le\mu(\al)\le\ta}c^{[\ta',\ta]}_\al y^\al.\]
Then
\[c^{[\ta',\ta]}_\al
=\sum_{\over{(\al_1,\dots,\al_k)\in\cP_\al}{
\mu(\al'_i)\ge\ta'\text{ and }\mu(\al-\al'_i)\le\ta\ \forall 1\le i<k
}}(-1)^{k}b_{\al_1}\dots b_{\al_k}.\]
\end{theorem}
\begin{proof}
If $c^{[\ta',\ta]}$ described by the last formula is inverse to $a^{[\ta',\ta]}$ for any $\ta'<\ta$, then we obtain from the equation
\[a^{\ta} a^{[\ta',\ta)}=a^{[\ta',\ta]},\]
that
\[c^{[\ta',\ta]}a^{\ta}=c^{[\ta',\ta)}.\]
Conversely, it is enough to prove that the $c^{[\ta',\ta]}$ satisfy this formula, as we know that $c^{[\ta',+\infty]}$ is inverse to $a^{[\ta',+\infty]}$.

Let us compare the coefficients of $y^\al$ on both sides. We can assume that $\ta'\le\mu(\al)\le\ta$. We will denote the sequences appearing in the definition of $c^{[\ta',\ta]}_\al$ by $\cP_\al'(\ta,\ta')$ and the sequences appearing in the definition of $c^{[\ta',\ta)}_\al$ by $\cP_\al'(\ta-,\ta')$.
Assume that a sequence $\ub\al=(\al_1,\dots,\al_k)\in\cP_\al$ (that is, the summand $b_{\al_1}\dots b_{\al_k}$) appears in the product on the left. There are three possibilities
\begin{enumerate}
	\item It is a summand of $1\cdot a^\ta$. In this case $\mu(\al)=\ta$ and $\mu(\al'_i)>\ta$ for $1\le i<k$. This implies that $\ub\al\in\cP_\al'(\ta,\ta')$.
	\item It is a summand of $c^{[\ta',\ta]}\cdot 1$. In this case $\ub\al\in\cP'_\al(\ta,\ta')$ automatically.
	\item It is a summand of $c^{[\ta',\ta]}_\be\cdot a^\ta_\ga$, where $\be+\ga=\al$. Let $1\le j<k$ be such that $\ub\be=(\al_1,\dots,\al_j)\in\cP_\be$ and $\ub\ga=(\al_{j+1},\dots,\al_k)\in\cP_\ga$. Then
$\ta'\le\mu(\be)\le\ta$, $\mu(\ga)=\ta$ and $\ub\ga$ is stable, that is, $\mu(\al-\al'_i)<\ta$ for $i>j$ and $\mu(\al-\al'_j)=\ta$. This means that such $j$ is unique. We claim that $\ub\al\in\cP'_\al(\ta,\ta')$. We know that $\mu(\al'_j)\ge\ta'$ and $\mu(\al'_i-\al'_j)>\ta>\ta'$ for $j<i<k$. This implies that $\mu(\al'_i)\ge\ta'$ for any $1\le i\le k$. We know that for $i>j$: $\mu(\al-\al'_i)<\ta$, $\mu(\al-\al'_j)=\ta$, and for $i<j$:
$\mu(\al'_j-\al'_i)\le\ta$ and therefore $\mu(\al-\al'_i)\le\ta$. This proves that $\ub\al\in\cP'_\al(\ta,\ta')$.
\end{enumerate}

Assume that $\mu(\al)=\ta$ and $\ub\al\in\cP_\al'(\ta,\ta')$. Then $\mu(\al-\al'_i)\le\ta$ and therefore also $\mu(\al'_i)\ge\ta$ for any $1\le i<k$. Assume that there is an equality for some $i$. Then $\ub\al$ appears in cases 2, 3 (with different signs), but does not appear in the first case. If there is a strict inequality for every $i$ then $\ub\al$ appears in cases 1, 2 (with different signs), but does not appear in the third case. In both situations the contribution of $\ub\al$ is zero.

Assume that $\mu(\al)<\ta$ and $\ub\al\in\cP_\al'(\ta,\ta')$. The first case cannot appear. We have $\mu(\al-\al'_i)\le\ta$ for $1\le i<k$. Assume that there is an equality for some $i$.
Let $j$ be the maximal index with this property. Then $(\al_1,\dots,\al_j)\in\cP'(\ta,\ta')$ and $(\al_{j+1},\dots,\al_k)$ is stable having slope \ta. 
Then $\ub\al$ appears in cases 2, 3 (with different signs) and its contribution is zero. If there is a strict inequality for every $i$ then $\ub\al$ appears just in case 2. Moreover, $\ub\al$ belongs to $\cP_\al(\ta-,\ta')$.
\end{proof}

\begin{corollary}
\label{crl:invert ray}
Let $a^{\ta}=1+\sum_{\mu(\al)=\ta}a_\al y^\al$ and \[(a^{\ta})\inv=c^{\ta}=1+\sum_{\mu(\al)=\ta}c^{\ta}_\al y^\al.\]
Then, for any $\al$ with $\mu(\al)=\ta$,
\[c^{\ta}_\al
=\sum_{\over{(\al_1,\dots,\al_k)\in\cP_\al}{
\mu(\al'_i)\ge\ta\ \forall 1\le i<k
}}(-1)^{k}b_{\al_1}\dots b_{\al_k}=c^{\le\ta}_\al=c^{\ge\ta}_\al.\]
\end{corollary}

\subsection{Slice formula for curves}
We consider the lattice $\Ga=\bZ^2$ and the semigroup
\[\Ga_+=\sets{(r,d)\in\Ga}{r\ge1}\]
with stability $\mu(r,d)=d/r$.
In contrast to the previous situation the elements in $\Ga_+$ could admit an infinite number of partitions. One simplification though will be that the elements $b_\al$, for $\al=(r,d)\in\Ga_+$, will be independent of $d$. Let $A$ be a commutative ring and let 
\[(a_{\al})_{\al\in\Ga_+},\qquad (b_r)_{r\ge1}\]
be two families of elements in $A\lser x$ such that for any $\al=(r,d)$ \cite[Eq.24]{zagier_elementary}
\[b_r=\sum_{\over{\al_1+\dots+\al_k=\al}{\mu(\al_1)>\dots>\mu(\al_k)}}x^{-\oh\sum_{i<j}\ang{\al_i,\al_j}}a_{\al_1}\dots a_{\al_k},\]
where $\ang{(r,d),(r',d')}=2(rd'-r'd)$.

\begin{remark}
In order to relate this equation to \eqref{eq:b as a}, we consider a non-commutative ring $R=A\lser x[\Ga_+]$ with multiplication
\[y^\al\circ y^\be=x^{-\oh\ang{\al,\be}}y^{\al+\be},\qquad \al,\,\be\in\Ga_+\]
and elements
\[\ub a_\al=a_\al y^\al,\qquad \ub b_\al=b_r y^\al\]
for $\al=(r,d)\in\Ga_+$. Then the above equation implies
\[\ub b_\al=\sum_{\over{\al_1+\dots+\al_k=\al}{\mu(\al_1)>\dots>\mu(\al_k)}}\ub a_{\al_1}\circ \dots\circ \ub a_{\al_k}\]
which is equivalent to \eqref{eq:b as a}.
\end{remark}

For any $\ta\in\bR$ define new elements
\[a^{\le\ta}_\al=\sum_{\over{\al_1+\dots+\al_k=\al}{\ta\ge\mu(\al_1)>\dots>\mu(\al_k)}}x^{-\oh\sum_{i<j}\ang{\al_i,\al_j}}a_{\al_1}\dots a_{\al_k}\]
and similarly define $a^{\ge\ta}_\al$, $a^{<\ta}_\al$, $a^{>\ta}_\al$.
Applying the previous results we can give a new proof of \cite[Theorem~3]{zagier_elementary}.

\begin{theorem}
We have
\begin{enumerate}
	\item For $\mu(\al)\le\ta$ (we define $r'_i=r_1+\dots+r_i$)
\[a^{\le\ta}_\al=\sum_{r_1+\dots+r_k=r}(-1)^{k-1}b_{r_1}\dots b_{r_k}
x^{-(r-r_k)d}\prod_{i=1}^{k-1}\frac{x^{(r_i+r_{i+1})(1+\flr{r'_i\ta})}}{1-x^{r_i+r_{i+1}}}.\]
	\item For $\mu(\al)\ge\ta$
\[a^{\ge\ta}_\al=\sum_{r_1+\dots+r_k=r}(-1)^{k-1}b_{r_1}\dots b_{r_k}
x^{(r-r_k)d}\prod_{i=1}^{k-1}\frac{x^{(r_i+r_{i+1})(1-\ceil{r'_i\ta})}}{1-x^{r_i+r_{i+1}}}.\]
	\item For $\mu(\al)<\ta$
\[a^{<\ta}_\al=\sum_{r_1+\dots+r_k=r}(-1)^{k-1}b_{r_1}\dots b_{r_k}
x^{-(r-r_k)d}\prod_{i=1}^{k-1}\frac{x^{(r_i+r_{i+1})\ceil{r'_i\ta}}}{1-x^{r_i+r_{i+1}}}.\]
	\item For $\mu(\al)>\ta$
\[a^{>\ta}_\al=\sum_{r_1+\dots+r_k=r}(-1)^{k-1}b_{r_1}\dots b_{r_k}
x^{(r-r_k)d}\prod_{i=1}^{k-1}\frac{x^{-(r_i+r_{i+1})\flr{r'_i\ta}}}{1-x^{r_i+r_{i+1}}}.\]
\end{enumerate}
\end{theorem}
\begin{proof}
We will prove just the first equation.
Using Theorem \ref{th:slices} we obtain
\begin{multline}
\label{eq:finite to zagier}
a^{\le\ta}_\al
=\sum_{\over{(\al_1,\dots,\al_k)\in\cP_\al}{\mu(\al'_i)>\ta\ \forall i<k,\,\al_i=(r_i,d_i)}}(-1)^{k-1}x^{-\oh\sum_{i<j}\ang{\al_i,\al_j}}b_{r_1}\dots b_{r_k}\\
=\sum_{r_1+\dots+r_k=r}(-1)^{k-1}b_{r_1}\dots b_{r_k}
\sum_{\over{d_1+\dots+d_k=d}{\mu(\al'_i)>\ta\ \forall i<k}}x^{-\oh\sum_{i<j}\ang{\al_i,\al_j}}.
\end{multline}
Let $(r_1,\dots,r_k)$, $(d_1,\dots,d_k)$ be two sequences and let $r=\sum_{i=1}^k r_i$, $d=\sum_{i=1}^k d_i$, $d'_i=d_1+\dots+d_i$. Then
\[\sum_{i<j}(r_id_j-r_jd_i)
=\sum_{i=1}^{k-1}(r_id-(r_i+r_{i+1})d'_i)
=(r-r_k)d-\sum_{i=1}^{k-1}(r_i+r_{i+1})d'_i.\]
In particular,
\[\oh\sum_{i<j}\ang{\al_i,\al_j}
=\sum_{i<j}(r_id_j-r_jd_i)
=(r-r_k)d-\sum_{i=1}^{k-1}(r_i+r_{i+1})d'_i.\]
For a given sequence $(r_1,\dots,r_k)$, we have
\begin{align*}
\sum_{\over{d_1+\dots+d_k=d}{\mu(\al'_i)>\ta\ \forall i<k}}x^{-\oh\ang{\al_i,\al_j}}
&=x^{-(r-r_k)d}\sum_{\over{d'_1,\dots,d'_{k-1}}{d'_i>r'_i\ta}}x^{\sum_{i=1}^{k-1}(r_i+r_{i+1})d'_i}\\
&=x^{-(r-r_k)d}\prod_{i=1}^{k-1}\frac{x^{(r_i+r_{i+1})(1+\flr{r'_i\ta})}}{1-x^{r_i+r_{i+1}}},
\end{align*}
where we applied the formula
\[\sum_{d>\ta}x^d=\frac{x^{\flr \ta+1}}{1-x}.\]
\end{proof}

\begin{remark}
Note that
\[
\sum_{i=1}^{k-1}(r_i+r_{i+1})(1+\flr{r'_i\ta})
=\sum_{i=1}^{k-1}(r_i+r_{i+1})(\ang{r'_i\ta}+r'_i\ta)
=\sum_{i=1}^{k-1}(r_i+r_{i+1})\ang{r'_i\ta}+(r-r_k)\ta r,
\]
where $\ang u=1+\flr u-u$. This implies
\[a^{\le\ta}_\al=\sum_{r_1+\dots+r_k=r}(-1)^{k-1}b_{r_1}\dots b_{r_k}
x^{(r-r_k)(\ta r-d)}\prod_{i=1}^{k-1}\frac{x^{(r_i+r_{i+1})\ang{(r'_i)\ta}}}{1-x^{r_i+r_{i+1}}}.\]
This result was originally proved by Zagier \cite[Theorem 3]{zagier_elementary} by different methods.
\end{remark}

\begin{remark}
If $\ta=\mu(\al)=d/r$ then $a^{\le\ta}_\al=a_\al$. Therefore
\[a_\al=\sum_{r_1+\dots+r_k=r}(-1)^{k-1}b_{r_1}\dots b_{r_k}
x^{-(r-r_k)d}\prod_{i=1}^{k-1}\frac{x^{(r_i+r_{i+1})(1+\flr{r'_id/r})}}{1-x^{r_i+r_{i+1}}}.\]
\end{remark}

\begin{remark}
We can actually write an explicit formula for $a^{[\ta',\ta]}_\al$. To do this we should consider in the proof of the theorem an additional condition $\mu(\al-\al'_i)<\ta'$ which is equivalent to
$d'_i>d-(r-r'_i)\ta'$. Therefore we should require
\[d'_i>\min\set{r'_i\ta,d+(r'_i-r)\ta'}.\]
Repeating the arguments of the theorem we obtain
\[a^{[\ta',\ta]}_\al=\sum_{r_1+\dots+r_k=r}(-1)^{k-1}b_{r_1}\dots b_{r_k}
x^{-(r-r_k)d}\prod_{i=1}^{k-1}\frac{x^{(r_i+r_{i+1})(1+\flr{\min\set{r'_i\ta,d+(r'_i-r)\ta'}})}}{1-x^{r_i+r_{i+1}}}.\]
\end{remark}

\begin{remark}\label{rk811}
Using the variable $q=x\inv$, we can write
\begin{equation}
b_r=\sum_{\over{\al_1+\dots+\al_k=\al}{\mu(\al_1)>\dots>\mu(\al_k)}}(-q^\oh)^{\sum_{i<j}\ang{\al_i,\al_j}}a_{\al_1}\dots a_{\al_k},
\end{equation}
\begin{equation}
a^{\ge\ta}_\al=\sum_{r_1+\dots+r_k=r}b_{r_1}\dots b_{r_k}
q^{-(r-r_k)d}\prod_{i=1}^{k-1}\frac{q^{(r_i+r_{i+1})\ceil{r'_i\ta}}}{1-q^{r_i+r_{i+1}}}.
\label{eq:ge ta}
\end{equation}
If $\ta=\mu(\al)$ then $a_\al=a^{\ge\ta}_\al=a^{\le\ta}_\al$. This implies
\begin{align}
a_\al
&=\sum_{r_1+\dots+r_k=r}b_{r_1}\dots b_{r_k}
q^{-(r-r_k)d}\prod_{i=1}^{k-1}\frac{q^{(r_i+r_{i+1})\ceil{r'_id/r}}}{1-q^{r_i+r_{i+1}}}\\
&=\sum_{r_1+\dots+r_k=r}b_{r_1}\dots b_{r_k}
q^{(r-r_k)d}\prod_{i=1}^{k-1}\frac{q^{-(r_i+r_{i+1})\flr{r'_id/r}}}{1-q^{r_i+r_{i+1}}}.\nonumber
\end{align}
\end{remark}

\subsection{Inversion formula for curves}
In the same way as in section \ref{subs:inversion} we can define the elements $c^{\le\ta}_\al$, $c^{\ge\ta}_\al$, $c^{<\ta}_\al$, $c^{>\ta}_\al$ in the case of curves. For example
\[1+\sum_{\mu(\al)\le\ta}c^{\le\ta}_\al y^\al=\rbr{1+\sum_{\mu(\al)\le\ta}a^{\le\ta}_\al y^\al}\inv.\]

\begin{theorem}
We have
\begin{enumerate}
	\item For $\mu(\al)\le\ta$
\[c^{\le\ta}_\al=\sum_{r_1+\dots+r_k=r}(-1)^{k}b_{r_1}\dots b_{r_k}
x^{(r-r_k)d}\prod_{i=1}^{k-1}\frac{x^{-(r_i+r_{i+1})\flr{r'_i\ta}}}{1-x^{r_i+r_{i+1}}}.\]
	\item For $\mu(\al)\ge\ta$
\[c^{\ge\ta}_\al=\sum_{r_1+\dots+r_k=r}(-1)^{k}b_{r_1}\dots b_{r_k}
x^{-(r-r_k)d}\prod_{i=1}^{k-1}\frac{x^{(r_i+r_{i+1})\ceil{r'_i\ta}}}{1-x^{r_i+r_{i+1}}}.\]
	\item For $\mu(\al)<\ta$
\[c^{<\ta}_\al=\sum_{r_1+\dots+r_k=r}(-1)^{k}b_{r_1}\dots b_{r_k}
x^{(r-r_k)d}\prod_{i=1}^{k-1}\frac{x^{(r_i+r_{i+1})(1-\ceil{r'_i\ta})}}{1-x^{r_i+r_{i+1}}}.\]
	\item For $\mu(\al)>\ta$
\[c^{>\ta}_\al=\sum_{r_1+\dots+r_k=r}(-1)^{k}b_{r_1}\dots b_{r_k}
x^{-(r-r_k)d}\prod_{i=1}^{k-1}\frac{x^{(r_i+r_{i+1})(1+\flr{r'_i\ta})}}{1-x^{r_i+r_{i+1}}}.\]
\end{enumerate}
\end{theorem}
\begin{proof}
We will prove just the last equation.
In the same way as in Eq.~\eqref{eq:finite to zagier}, the coefficient of $(-1)^kb_{r_1}\dots b_{r_k}$ in $c^{>\ta}_\al$ is
\[\sum_{\over{d_1+\dots+d_k=d}{\mu(\al'_i)>\ta\ \forall i<k}}x^{-\oh\sum_{i<j}\ang{\al_i,\al_j}}
=x^{-(r-r_k)d}\prod_{i=1}^{k-1}\frac{x^{(r_i+r_{i+1})(1+\flr{r'_i\ta})}}{1-x^{r_i+r_{i+1}}}.\]
\end{proof}

\begin{remark}
Let $(1+\sum_{\mu(\al)=\ta}a_\al y^\al)=1+\sum_{\mu(\al)=\ta}c_\al y^\al$. Then, according to Corollary \ref{crl:invert ray}, we have $c_\al=c^{\le\ta}_\al=c^{\ge\ta}_\al$ for any \al with $\mu(\al)=\ta$. 
\end{remark}

\begin{remark}\label{rk814}
Using the variable $q=x\inv$, we can write
\begin{equation}
c^{>\ta}_\al=-\sum_{r_1+\dots+r_k=r}b_{r_1}\dots b_{r_k}
q^{(r-r_k)d}\prod_{i=1}^{k-1}\frac{q^{-(r_i+r_{i+1})\flr{r'_i\ta}}}{1-q^{r_i+r_{i+1}}}.
\label{eq:>ta inv}
\end{equation}
\end{remark}

\providecommand{\bysame}{\leavevmode\hbox to3em{\hrulefill}\thinspace}
\providecommand{\href}[2]{#2}

\end{document}